\documentclass{amsart}
\usepackage{amsmath, amsthm, amssymb,esint}
\usepackage{fullpage}
\usepackage{color}
\usepackage{hyperref}
\usepackage{soul}
\usepackage{enumitem}
\usepackage{tikz, pgfplots}
\usepackage{url}
\usepackage{cancel}
\usetikzlibrary{positioning}
\usepackage{comment}
\usepackage{xcolor}

\includecomment{comment}

\usepackage{pgfplots}

\pgfplotsset{compat=1.11}
\usetikzlibrary{patterns}
\usetikzlibrary{optics}

\definecolor{ForestGreen}{RGB}{34,139,34}

\newtheorem{theorem}{Theorem}

\newtheorem{lemma}{Lemma}
\newtheorem{proposition}[lemma]{Proposition}

\newtheorem{remark}[lemma]{Remark}

\newtheorem{question}{Question}
\numberwithin{lemma}{section}

\newcommand{\la}{\langle}
\newcommand{\ra}{\rangle}

\DeclareMathOperator{\supp}{\mathrm{supp}}
\DeclareMathOperator{\Char}{\mathrm{char}}

\newcommand{\blue}[1]{{\color{blue}#1}}

\numberwithin{equation}{section}

\newcommand{\R}{{\mathbb R}}
\newcommand{\N}{{\mathbb N}}

\renewcommand{\R}{\mathbb R}

\newcommand{\tgamma}{{\tilde \gamma}}

\newcommand{\uu}{{\mathbf u}}

\newcommand{\qq}{{\mathbf q}}
\newcommand{\ww}{{\mathbf w}}

\newcommand{\rr}{{\mathbf r}}

\newcommand{\sgn}{\mathop{\mathrm{sgn}}}

\newcommand{\tchi}{{\tilde \chi}}

\newcommand{\tL}{\tilde L}

\newcommand{\Diag}{{\mathfrak{D}}}
\begin{document}

\title{Testing by wave packets and modified scattering in nonlinear dispersive pde's}

\author{Mihaela Ifrim}
\address{Department of Mathematics, University of Wisconsin, Madison}
\email{ifrim@wisc.edu}

\author{ Daniel Tataru}
\address{Department of Mathematics, University of California at Berkeley}
\email{tataru@math.berkeley.edu}

\begin{abstract}
Modified scattering phenomena are encountered in the study of global
properties for nonlinear dispersive partial differential equations in situations where 
the decay of solutions at infinity is borderline and scattering 
fails just barely. An interesting example is that of problems 
with cubic nonlinearities in one space dimension.

The method of testing by wave packets was introduced by the authors
as a tool to efficiently capture the asymptotic equations 
associated to such flows, and thus establish the modified 
scattering mechanism in a simpler, more efficient fashion, and at lower regularity. In these expository notes we describe how this method  can be applied to problems with general dispersion relations.
\end{abstract}


\maketitle

\setcounter{tocdepth}{1}
\tableofcontents

\section{Introduction}

Given a nonlinear perturbation of a linear partial differential equation, scattering theory aims to compare 
the long time dynamics of the nonlinear problem with the long time dynamics of the 
corresponding linear flow. This is particularly interesting in the context of dispersive equations, which have two key properties:
\begin{itemize}
    \item a conservative nature at the $L^2$ level, with some energy that is either exactly 
    conserved, or essentially conserved for small data.
    
    \item  some form of uniform or averaged decay, whose effect is that the strength of the nonlinear interactions decays with time.
\end{itemize}

Whether scattering holds for a given problem, that depends on the relative 
strength of the nonlinearity on one hand, and on the dispersive effects on the other hand.
If the nonlinearity is mild, then scattering holds, in the sense that, as time goes to infinity, the solutions to the nonlinear problem will approach solutions to the linear problem.

In this paper we are instead interested in the \emph{modified scattering} phenomena. These  are encountered in situations where  the decay of solutions at infinity is borderline, and scattering fails just barely. Then one might expect that the nonlinear asymptotic behavior 
can be seen as some perturbation of the linear asymptotic behavior. Such dynamics are encountered for many classes of equations, and the modified scattering effects 
may vary from case to case; this may include for instance corrections to the velocity, amplitude or phase.
The class of problems we are interested in here is that of dispersive problems 
with cubic nonlinearities in one space dimension. As we will see, for this class  modified scattering means a phase correction on a logarithmic time scale. 

\emph{The method of testing by wave packets} was introduced by the authors
in the context of the cubic nonlinear Schr\"odinger  flow (NLS) \cite{NLS}, and used later
in water wave contexts \cite{IT-global}, \cite{IT-t}, \cite{AIT-global}
as a tool to efficiently capture the asymptotic equations associated to such flows, and thus establish the modified scattering mechanism.  
See also \cite{BHG,r1,r2,r3} for further examples
where this idea is used.

These notes, written by the authors for a summer school at MSRI~\cite{msri-summer} in 2020, and based on earlier set of notes prepared by  the authors for an AMS meeting in Las Vegas in 2016,  aim to describe how this method can be applied to problems with general dispersion relations. Notably, here we work with  with minimal structure assumptions on the nonlinearity, which do not include a scaling symmetry or energy conservation.

\subsection{ A model dispersive problem}
The model problem we consider here is a one dimensional evolution of the form
\begin{equation}\label{main}
\left\{
\begin{aligned}
&i \partial_t u - A(D) u = Q( u, \bar u, u )\\
&u(0,x) = u_0(x),
\end{aligned}
 \right.
\end{equation}
for a complex valued function 
\[
u: \R \times \R \to \mathbb{C} .
\]
Here we will make the following general assumptions:
\begin{enumerate}[label=(H\arabic*), ref=(H\arabic*)]

\item \emph{Real symbol.} The symbol $a(\xi)$ of the multiplier $A(D)$ is real and smooth;
this guarantees that the $L^2$ norm is preserved for solutions to the corresponding linear flow.

\item \emph{Dispersive character.} The group velocity depends on the frequency,
\[
\R \ni \xi \to a'(\xi) \in \R \quad \text{ is a diffeomorphism} \quad (i. e. \ \ a'' \neq 0).
\]
\item \emph{Cubic, translation and phase shift invariant nonlinearity}. The nonlinearity $Q$
is defined by its smooth symbol $q$ as follows:
\[
\widehat{Q( u, \bar u, u )}(\xi) = \frac{1}{2\pi} \int_{\xi = \xi_1 - \xi_2 + \xi_3} q(\xi_1,\xi_2,\xi_3) \hat u(\xi_1) 
\overline{\hat{u}(\xi_2)} \hat u(\xi_3)\,  d\sigma,
\]
with $d\sigma =d_{\xi_i}d_{\xi_j}$, $i\neq j=\overline{1,3}$ (where one needs to adjust the sign corresponding to the chosen $(i,j)$ pair). 

\item \emph{Conservative nonlinearity.}   The symbol associated to $Q$ and computed on the diagonal must be real i.e.
\[
q(\xi,\xi,\xi) \in \R.
\]

\end{enumerate}

For the Cauchy problem \eqref{main} we will ask the following question:

\begin{question}
Assume that the initial data for the evolution \eqref{main} is small, localized
and sufficiently smooth. Does this guarantee that we have global solutions with dispersive decay and modified scattering ? 
\end{question}

Our goal in this paper will be to show that the answer is affirmative, under  
minimal assumptions on the behaviour of the symbols $a (\cdot)$ and $q (\cdot, \cdot, \cdot)$ at infinity, and also under minimal regularity and decay assumptions for the initial data $u_0$.

\subsection{ An overview of the paper} To motivate the results, our exposition will begin with a brief discussion of linear dispersion in Section~\ref{s:linear}, which notably ends with the vector field bound in Proposition~\ref{p:vf}. The standard linear scattering mechanism described here serves as the basis for the 
nonlinear, modified scattering results which are the main goal of the paper.

In the following section, i.e. in Section ~\ref{s:asymptotic},  we provide a heuristic discussion of the modified scattering phenomena. The premise here is that, relative to the linear scattering mechanism, the nonlinear \emph{asymptotic profile} is governed by an \emph{asymptotic equation}. The objective is then to efficiently capture both of these objects in the analysis.
We outline several  ideas which have been used over the years, and finish with a brief introduction to wave packet testing. 

At this point we are ready to present the main results of the paper, which in a nutshell assert that global solutions 
with modified scattering dynamics exist for the flow in \eqref{main} under suitable assumptions. For expository purposes, we will split the discussion in two parts:

\begin{enumerate}[label=(\roman*)]
\item In Section~\ref{s:compact} we consider cubic forms $Q$ with compactly supported symbols. Then one may also assume that the 
solution $u$ has a compactly supported Fourier transform, 
and no restriction is imposed on the behavior
of $a$ at infinity. In this setting the arguments are simpler, 
and we are able to present the main steps, namely the energy estimates and the wave packet testing, in a streamlined 
fashion, without distracting technicalities. 

\item In Section~\ref{s:general} we consider cubic forms $Q$ with bounded symbols,
and, correspondingly, symbols $a$ so that $a''(\xi) \approx |\xi|^\sigma$ at infinity for some real $\sigma$.

Then we show that, for initial data $u_0$ which is small in suitable spaces
\[
u_0 \in H^{s_0}, \qquad x u_0 \in H^{s_1},
\]
the solutions are global and their asymptotic behavior is 
still governed by the modified scattering mechanism.  
Anticipating the precise results in Section~\ref{s:general}, we point out here that there are two qualitatively different scenarios:

\begin{description}
\item[ (i) The generalized NLS case] $\sigma \geq -1$, where $a$ is superlinear at infinity and thus we have infinite speed of propagation. This includes for instance

\begin{enumerate}[label=\alph*)]
 \item SQG type problems, $\sigma = -1$, where we take 
 \[
 s_0 = 0+, \qquad s_1=1.
 \]

\item NLS type problems, $\sigma = 0$, where we take 
\[
s_0=-\frac12+, \qquad s_1=\frac12.
\]
\item KdV type problems, $\sigma =1$, where we take 
\[
s_0 = -1+, \qquad s_1=0.
\]

\end{enumerate}

\item[(ii) The generalized Klein-Gordon case] $\sigma < -1$, 
where $a$ is linear at infinity and then we have finite speed of propagation in the high frequency limit. This includes for instance

\begin{enumerate}[label=\alph*)]
\item gravity wave models, $\sigma = -3/2$, where we take 
\[
s_0=1+, \qquad s_1=\frac12.
\]

\item  Klein-Gordon models, $\sigma = -3$, where we take 
\[
s_0=1+, \qquad s_1=2.
\]
\end{enumerate}
\end{description}

\end{enumerate}

While our results do allow for a full range of asymptotic behaviors for $a$, this is far from capturing a full range of problems,  as  $Q$ is not in general bounded in many interesting models. We leave it for the  interested reader 
to investigate more general situations. 

Another line of investigation which is completely omitted
in our discussion here is that of normal form methods,
which in many instances allow one to expand the scope 
of this type of results to problems which also have 
nonresonant quadratic nonlinearities.

\subsection{Acknowledgements} 
 The first author was supported by a Luce Professorship, by the Sloan Foundation, and by an NSF CAREER grant DMS-1845037. The second author was supported by the NSF grant DMS-2054975 as well as by a Simons Investigator grant from the Simons Foundation. Some of this work was presented during an MSRI Graduate summer school in  2020. Other parts of the work were
  carried out while both authors were 
 participating in the MSRI program ``Mathematical problems in fluid dynamics" during Spring 2021. 

 The authors are very grateful to the anonymous referee for a thorough reading of the paper, which led to many improvements and corrections.

\section{Dispersive decay for the linear equation}
\label{s:linear}

In this section we consider the dispersive properties and the asymptotic behavior of the solutions to the associated
linear problem,
\begin{equation}\label{linear}
\left\{
\begin{aligned}
&i \partial_t u - A(D) u = 0\\ &u(0) = u_0.
\end{aligned}
\right.
\end{equation}
To avoid distracting technicalities, we will assume that the initial data $u_0$ is frequency localized in a fixed compact set.

\subsection{The fundamental solution and dispersive decay}
Denoting by $\tau$, respectively $\xi$, the time and the space Fourier variables,
the symbol of the linear operator is
\[
p(\tau,\xi) = - \tau - a(\xi), \qquad P= i \partial_t - A(D),
\]
and its characteristic set is the graph of $-a$,
\[
\text{char } P = \{ \tau = -a(\xi) \}.
\]
This is commonly referred to as the \emph{dispersion relation}. 

The associated Hamilton flow is 
\[
(x,\xi) \to (x+t a'(\xi), \xi).
\]
In particular we note here the group velocity $a'(\xi)$, which depends on the spatial frequency $\xi$ of the waves. We denote 
the range of admissible group velocities by 
\[
V = a'(\R).
\]
Here we may distinguish two different scenarios, depending on the asymptotic behavior of $a$ at infinity.
\begin{enumerate}[label=\alph*)]
\item The generalized NLS case, where $a$ has superlinear growth at $\pm \infty$, in which case $a':\R \to \R$ is surjective 
so $V = \R$, i.e. waves 
propagate with all possible velocities.

\item The generalized Klein-Gordon case, where $a$ has linear growth at $\pm \infty$, in which case the set $V$ of possible group velocities  is a bounded open interval.
\end{enumerate}
Of course, one may also differentiate between the behavior of $a$ 
at $+\infty$ and at $-\infty$, with obvious consequences.

\medskip

The  spatial Fourier transform of the fundamental solution $K(t,x)$ for \eqref{linear} is given by
\[
\hat K(t,\xi) = e^{-it a(\xi)},
\]
which yields the following oscillatory integral representation for $K$:
\[
K(t,x) = \frac{1}{2\pi} \int_\R e^{i(x\xi - ta(\xi)) } \, d\xi.
\]

By the assumption (H2), the phase function is nondegenerate, and has at most one critical point,  namely the solution $\xi$ of  the equation
\[
 x - t a'(\xi) = 0.
\]
Denoting $v = x/t$ as the velocity along a ray starting from the origin, this becomes  
\begin{equation}\label{xiv}
    v = a'(\xi).
\end{equation}
We denote the solution  of \eqref{xiv} by $\xi_v$. In the generalized NLS case  this solution exists for all real velocities $v$. But in the generalized Klein-Gordon case  the critical point exists only for $v \in V$.

Assuming that $v \in V$, the asymptotics of the 
the fundamental solution along the ray $x=vt$ can be computed using the stationary phase method, see \cite{Stein}, which yields
the asymptotic expansion 
\begin{equation} \label{fundamental}
K(t,vt) \approx \frac{1}{\sqrt{2\pi t |a''(\xi_v)|}} e^{i t \phi(v)}e^{-\frac{i\pi \sgn(a'')}4} + O(t^{-1})
\end{equation}
where the phase function $\phi(v)$ is given by 
\[
\phi(v) = v \xi_v - a(\xi_v).
\]
This holds uniformly for $v$ in compact subsets of $V$, with a more complex behavior at the endpoints and rapid decay along rays outside the closure of $V$ in the generalized Klein-Gordon case.

Since by (H2) $a$ is either convex or concave, this last expression also allows 
one to interpret $\phi$ as the Legendre transform of $a$, so that $\phi'$ and $a'$ are inverse functions, 
\begin{equation}\label{phi-prime}
\xi_v = \phi'(v).
\end{equation}

Equivalently,  $\phi$ can be thought of as the solution to the eikonal equation
\begin{equation}\label{a-vs-phi}
a(\phi_v) =  \phi -v \phi_v.
\end{equation}

More generally, one may apply the stationary phase method to compute the asymptotics for any solution $u$ to the linear equation \eqref{linear} with initial data $u_0$ with a smooth 
Fourier transform (which corresponds to a localized initial data), namely
\begin{equation}\label{linear-scattering}
 u(t,x=vt) \approx \gamma(v) \frac{1}{\sqrt{ t |a''(\xi_v)|}} e^{i t \phi(v)}e^{-\frac{i\pi \sgn(a'')}4} + O(t^{-1}) ,
\end{equation}
where the asymptotic profile $\gamma = \gamma(v)$ depends on the initial data in a straightforward fashion,
\[
\gamma(v) = \hat u_0(\xi_v).
\]

The expansion in \eqref{linear-scattering} is uniform for $v$
in a compact subset of $V$, or if $\hat u_0$ is compactly supported
or at the very least has sufficient decay at infinity.

 One  may also think of the linear dispersive decay of solutions to \eqref{linear} in a translation invariant fashion. This 
is described by the following result:

\begin{theorem}\label{t:disp}
Assume that the conditions (H1), (H2) hold.
Then the  following translation invariant decay estimates hold for frequency localized solutions to \eqref{linear}:

\begin{itemize}
\item Dispersive bounds:
\begin{equation}\label{dispersive}
\| u(t) \|_{L^\infty} \lesssim t^{-\frac12} \| u(0)\|_{L^1}.
\end{equation}

\item Strichartz estimates: 
\begin{equation}\label{Strichartz}
\| u\|_{L^4 L^\infty} \lesssim \| u(0)\|_{L^2}.
\end{equation}
\end{itemize}
\end{theorem}
The dispersive bound follows from the pointwise decay of the (frequency localized) fundamental solution, see \eqref{fundamental}.
The Strichartz estimates \eqref{Strichartz} can be seen as a direct consequence of the dispersive
estimates \eqref{dispersive} and  Young's inequality, see \cite{GV} and \cite{KeTa}.

Similar bounds hold for problems with unlocalized data, provided one adds appropriate multiplier weights depending on the asymptotic behavior of $a''$ at infinity. Some details are provided in the last section of the paper.

\subsection{Dispersion via energy estimates}

As noted above, the standard proof of the dispersive bound \eqref{dispersive} is via the pointwise bounds for the frequency localized fundamental solution, which in turn follow from the method of stationary phase, see e.g. \cite{Stein}. However,  
there is also an alternative, more robust approach via energy estimates and the vector field method.  

Precisely, using the atomic structure of the $L^1$ space, it suffices to prove the $t^{-\frac12}$ decay in \eqref{dispersive} for initial data $u_0$ which is both frequency localized and in the Schwartz space. To measure the decay,  we introduce the linear operator 
\begin{equation}\label{def:L}
L = x - t A'(D),
\end{equation}
which is the pushforward of $x$ along the corresponding linear flow.
For Schwarz data, we control
\begin{equation}
\|u_0\|_{L^2} + \| L u_0\|_{L^2}     \lesssim 1.
\end{equation}
Then we want to show that 
\begin{equation}\label{disp-decay-lin}
\|u(t)\|_{L^\infty} \lesssim t^{-\frac12}.
\end{equation}

To prove this, we observe that if $u$ solves \eqref{linear} then $Lu$ also solves \eqref{linear}.
Hence, using the conservation of the $L^2$ norm, it follows that
\begin{equation}\label{unif-en}
\|u(t)\|_{L^2} + \| L u(t)\|_{L^2}     \lesssim 1.
\end{equation}
Hence, one can think of the decay bound \eqref{disp-decay-lin} as a consequence of a  Sobolev-type interpolation bound, where the uniform norm for $u$ is estimated in terms of the uniform energy bound in \eqref{unif-en}. For later use, we will state and prove a more general statement,
where uniform dispersion is assumed globally and implicit constants are carefully controlled:

\begin{proposition}\label{p:vf}
Assume that the symbol $a(\cdot)$ satisfies the bounds
\begin{equation}
a''\approx R, \qquad |a'''| \lesssim MR,
\end{equation}
where $M$ and $R$ are  positive constants.
Then the  following estimate holds for any frequency localized function $u$:
\begin{equation}\label{vf-Sobolev}
\| u(t)\|_{L^\infty}^2 \lesssim \frac{1}{tR} (\| u\|_{L^2} \| Lu \|_{L^2}+M\|u\|_{L^2}^2).
\end{equation}
\end{proposition}

In particular, this yields the following result in the context of  
Theorem~\ref{t:disp}.

\begin{proposition}\label{p:vf-comp}
Assume that the conditions $(H1)$, $(H2)$ hold.
Then the  following estimate holds for any
function $u$  which is frequency localized in a fixed compact set:
\begin{equation}\label{vf-Sobolev-cor}
\| u(t)\|_{L^\infty}^2 \lesssim t^{-1} (\| u\|_{L^2} \| Lu \|_{L^2}+\|u\|_{L^2}^2)
\end{equation}
\end{proposition}
As discussed above, this in turn implies \eqref{dispersive}. Conversely, we remark that the above inequality can be obtained as a direct consequence of  \eqref{dispersive}.

\begin{proof}[Proof of Proposition~\ref{p:vf}]
We begin by observing that, by independently scaling the space and the time, both 
$R$ and $M$ can be seen as scaling parameters.  Precisely, a linear change of coordinates $x \to Mx$ reduces the problem to the case when $M=1$. Thus without any restriction in generality  we will assume $R=1$, $M=1$.

A second observation is that we can regularize the symbol $a$ on the $\sqrt{t}$ scale,
\[
\tilde a = P_{< t^{\frac12}}( D_\xi) a
\]
The bound on $a'''$ shows that 
\[
|a' - \tilde a'| \lesssim t^{-1}
\]
so that $a$ and $\tilde a$ can be used interchangeably in the Proposition.
The advantage is that $\tilde a$ satisfies higher regularity bounds
\begin{equation}\label{a-high}
|\tilde a^{(j)}| \lesssim t^{\frac{j-3}2}, \qquad j \geq 3.    
\end{equation}
From here on we will drop the $\tilde a$ notation and 
assume that $a$ satisfies \eqref{a-high}.

Next we introduce a secondary operator $\tL$, which can be used interchangeably  with $L$ in energy estimates. We recall that the symbol of $L$ is
\[
\ell(x,\xi) = x - t a'(\xi),
\]
so its characteristic set is given by 
\[
\Char L = \{ a'(\xi) = x/t \}.
\]
Using the property \eqref{phi-prime}, this can be written as 
\[
\Char L = \{ \xi = \phi'(x/t) \}.
\]

This leads us to define the operator
\[
\tL := t (\partial_x - i \phi'(x/t))  
\]
which has symbol
\[
\tilde \ell(x,\xi) = i t (\xi - \phi'(x/t)),
\]
and thus the key property that it has the same characteristic set as $L$. Since $a'$ and $\phi'$ are inverse functions, it follows directly that 
\[
 \phi'' \approx 1.
\]
In addition, from \eqref{a-high} and differentiation rules 
one also obtains
\begin{equation}\label{phi-high}
|\phi^{(j)}| \lesssim t^{\frac{j-3}2}, \qquad j \geq 3 .   
\end{equation}

To compare the two operators $L$ and $\tL$ we need the following G\"arding 
type inequality

\begin{lemma}\label{l:tL}
We have 
\[
\| \tL u \|_{L^2} \lesssim \|Lu\|_{L^2} +  \|u\|_{L^2}
.\]
\end{lemma}
The conclusion of the lemma is a direct consequence of the  corresponding symbol bound
\[
|\tilde{\ell}(x,\xi)|\lesssim |\ell(x,\xi)|
\]
via a semiclassical form of G\"arding's inequality. Precisely, we will directly invoke \cite[Theorem~3]{T-FF}, with the semiclassical parameter $\mu = t^{-\frac12}$. For convenience, we recall it here:

\begin{theorem}{\cite[Theorem~3]{T-FF}}    \label{t:garding}
Let $\mu > 0$ and $a_j,b$ be real symbols which satisfy  \begin{equation}
  \label{a2}
  |\partial_x^\alpha \partial_\xi^\beta a(x,\xi)| \leq c_{\alpha\beta} \mu^{\frac{|\alpha|-|\beta|}2},
\ \ \ \ |\alpha|+|\beta|\geq 2 
\end{equation}
so that
$ |b| \leq  \sum |a_j| $. Then 
\begin{equation}
\|B^w(x,D) u\|_{L^2} \lesssim \sum 
\|A_j^w(x,D) u\|_{L^2} + \|u\|_{L^2}\,.
\label{garding}\end{equation}
\end{theorem}

The fact that this is stated in the Weyl calculus makes no difference in our context. We also refer the reader to \cite{Delort-KG}, where a similar bound is derived using semiclassical calculus.

Now we return to the proof of Proposition~\ref{p:vf}, where
as discussed above we may substitute $L$ by $\tilde L$. Then
all we need is a simple integration, based on the relation
\[
\frac{d}{dx} |u|^2 = \frac2t \Im ( \bar{u} \tL u) .
\]
This yields 
\[
\|u\|_{L^\infty}^2 \lesssim \frac1t \|u\|_{L^2} \|\tL u\|_{L^2} ,
\]
thus completing the proof.

\end{proof}

We conclude this section with one last observation, which is that,
if $u$ is assumed to be frequency localized in some compact interval $I$,
then both the $L^2$ and the pointwise bounds for $u$ are better outside a neighbourhood of the the velocity range $J = a'(I)$. The following result clarifies the proper localization scales. 
The analysis is identical to the left and to the right of $J$. 
Hence, in order to set the notations, we fix a frequency $\xi_0$, the corresponding group velocity 
$v_0 = a'(\xi_0)$ and the associated position at time $t$, $x_0 = t v_0$. We also assume without any loss of generality that $a$ is convex.

\begin{proposition}\label{p:vf-ell}
In the context of Proposition~\ref{p:vf}, assume in addition that
$u$ is frequency localized in $I = \{ \xi< \xi_0\}$. Then we have better bounds for 
$u$ outside $J = \{ x < x_0\}$, as follows:

a) $L^2$ bounds:
\begin{equation}\label{vf-ell-L2}
\| (x-x_0)_+ u\|_{L^2} \lesssim \|Lu\|_{L^2} + M \|u\|_{L^2},    
\end{equation}

b) $L^\infty$ bounds
\begin{equation}\label{vf-ell-Linf}
|u(x)|^2 \lesssim \frac{1}{|x-x_0| R t}  ( \|Lu\|_{L^2} + M \|u\|_{L^2})^2, \qquad  
x > x_0.
\end{equation}
\end{proposition}
We remark that the bounds in the proposition are only interesting in the region
$\{ x - x_0 > (Rt)^\frac12 \}$. Closer to $x_0$, and in effect in the full region
$\{ |x - x_0| \lesssim (Rt)^\frac12 \}$ they can be replaced by 
\begin{equation}\label{vf-ell-near}
|u(x)|^2 \lesssim \frac{1}{ (R t)^\frac32}  ( \|Lu\|_{L^2} + M \|u\|_{L^2})^2, \qquad  
|x - x_0|  \lesssim (Rt)^\frac12.
\end{equation}

\begin{proof}
We first note that $R$ and $M$ in the hypothesis of the proposition are scaling parameters,
and we can simply set them equal to $1$.

Let $\chi=\chi(x)$ be a spatial cutoff function supported outside $J$, smooth on the $r$ scale and equal to $1$
in $[x_0+r,\infty)$, where
$r > 0$ is a parameter which will be chosen later as $r = t^{\frac12}$. We will establish the stronger $L^2$ bound
\begin{equation}\label{loc}
\| (x-x_0) \chi u \|_{L^2} + t \| \chi (D - \xi_0)  u \|_{L^2} \lesssim \|Lu\|_{L^2} + \|u \|_{L^2} 
\end{equation}
for $u$ frequency localized in $\xi < \xi_0$. To start with, we verify 
that this implies the bounds in the proposition. This is immediate for 
$x > x_0+r$, so we need to cover the remaining range. It suffices to show that
\[
|u(x)| \lesssim t^{-\frac34}(\|Lu\|_{L^2} + \|u \|_{L^2}), \qquad |x-x_0| \lesssim r,
\]
which is in effect exactly the bound \eqref{vf-ell-near}.

We already know this if $x-x_0 \approx r$. To capture the remaining range  we recall that we can use the operators $L$ and $\tL$ interchangeably in these bounds. Our starting point is the straightforward relation
\[
\left|\frac{d}{dx} |u|\right|
\lesssim \left|\frac{d}{dx} e^{-it \phi(x/t)} u\right|
= \frac{1}{t} |\tL u|.
\]
Then for $x_1$ in the full range $|x_1-x_0| \lesssim r$ we write by the fundamental theorem of calculus
\[
||u(x_1)| - |u(x)|| \lesssim t^{-1} \int_{x_1}^x |\tL u| dy \lesssim t^{-1} |x_1-x|^\frac12
\|\tL u\|_{L^2} \lesssim t^{-\frac34},
\]
which suffices. It remains to prove \eqref{loc}.

First, using the form of the operator $\tL$, we write
\[
-i\tL = t(D-\xi_0) - t(\phi'(x/t) -\phi'(x_0/t))  
\]
and estimate
\[
t \| \chi(D - \xi_0)  u \|_{L^2} \lesssim \| (x-x_0) \chi u \|_{L^2} + \|\tL u\|_{L^2}
\]
and thus reduce the bound \eqref{loc} to 
\begin{equation}\label{loc-a}
\| (x-x_0) \chi u \|_{L^2}  \lesssim \|u \|_{L^2} + \|\tL u\|_{L^2}.
\end{equation}
The next step is to discard the frequency localization, by adding a term on the right 
\begin{equation}\label{loc-b}
\| (x-x_0) \chi u \|_{L^2}  \lesssim \|u \|_{L^2} + \|t (D-\xi_0)\eta(D-\xi_0) u\|_{L^2} + \|\tL u\|_{L^2},
\end{equation}
where $\eta$ selects the region $\xi-\xi_0 > \rho$, 
with the parameter $\rho > 0$ to be chosen later 
of size $\rho \approx t^{-\frac12}$.

Here we use again the G\"arding type inequality \eqref{garding} with $\mu = t^{-\frac12}$.
At the symbol level, we need to verify that
\begin{equation}\label{loc-c}
(x-x_0) \chi  \lesssim t (\xi-\xi_0) \eta(\xi-\xi_0)
+ |x - t a'(\xi)|.
\end{equation}
If $\xi > \xi_0 + \rho$ then $\eta = 1$ and the inequality is directly verified without $\chi$. Else, for $x > x_0+r$ and $\xi < \xi_0+\rho$
we have 
\[ 
x - t a'(\xi) > x-x_0  - C t \rho  > r -Ct\rho, \qquad C = \sup_\xi  a''(\xi),
\]
which suffices provided that $ r \geq 2Ct\rho$.

It remains to ensure that we have the correct symbol regularity,
as required by Theorem~\ref{t:garding} with $\mu = t^{-\frac12}$.
This is indeed the case provided that 
\[
r^{-1} \lesssim t^{-\frac12}, \qquad  \rho^{-1} \lesssim t^\frac12  .
\]
To satisfy all of the above requirements it suffices to choose 
$r$ and $\rho$ so that
\[
r \approx t^{\frac12}, \qquad \rho \approx t^{-\frac12}, \qquad r \geq 2Ct\rho.  
\]
Then the $L^2$ bound \eqref{loc-b} follows from the symbol bound \eqref{loc-c}, and the proof of the proposition is complete. 
\end{proof}

\section{The asymptotic equation for the nonlinear problem}
\label{s:asymptotic}

We begin the discussion by recalling the asymptotic 
behavior of solutions for the linear flow \eqref{linear},
namely 
\begin{equation}\label{wish}
u(t,x) \approx \frac{1}{\sqrt{t}}\gamma(v) e^{i t \phi(v)}, 
\end{equation}
and ask whether such a pattern is also possible 
for the nonlinear flow \eqref{main}. This would require
the cubic term in the equation to play a perturbative 
role near infinity. 

However, a heuristic computation shows that this cannot happen. To see that, suppose, 
more generally, that for the solution $u$ we have 
an asymptotic representation of the form
\[
u(t,x) \approx \frac{1}{\sqrt{t}} \gamma(t,v) e^{i t \phi(v)}, 
\]
where $\gamma$ is a smooth function of $v$, uniformly in $t$. Then, at $(t,x)$, this solution has spatial frequency close to  
\[
\xi_v = \phi'(v).
\]
Expanding the symbol for $A$ in a Taylor series around $\xi = \xi_v$ we obtain
\[
a(\xi) = a(\xi_v) +  a'(\xi_v)(\xi-\xi_v)+ 
\frac12 
a''(\xi_v)(\xi-\xi_v)^2 + O(\xi-\xi_v)^3,
\]
which at the operator level yields
\[
A(D) = a(\xi_v) +  a'(\xi_v)(i \partial_x -\xi_v)+ 
\frac{i}2 a''(\xi_v) \phi''(v) + O(t^{-2}) + O(t^{-1}) (i \partial_x -\xi_v) + O(1)  (i \partial_x -\xi_v)^2.
\]
This further simplifies since $\phi'$ and $a'$ are inverse functions, so $a''(\xi_v) \phi''(v)=1$. Hence we obtain the semiclassical formula 
\[
A(D) u \approx t^{-\frac12} e^{i t \psi(v)}\left( a(\xi_v)\gamma(t,v) + i t^{-1}  a'(\xi_v) \gamma'(t,v) + \frac{i}{2} t^{-1}  + O(t^{-2})\right).
\]
A similar but simpler computation shows that 
\[
Q(u,\bar u, u) = t^{-\frac12} e^{i t \psi(v)} 
\left( t^{-1} q(\xi_v,\xi_v,\xi_v)  \gamma(t,v) |\gamma(t,v)|^2 + O(t^{-2})\right).
\]
Finally, by chain rule,
\[
i \partial_t u = t^{-\frac12} e^{i t \psi(v)} \left(
-\gamma(t,v)(\phi(v) - v \phi'(v)) - i t^{-1} v \gamma'(v) - \frac{i}{2} t^{-1} + i \gamma_t\right). \]
Substituting the last three relations into the equation, we 
cancel the leading terms using the eikonal equation \eqref{a-vs-phi}
(which justifies the phase in our ansatz in the first place),
and the $\gamma'$ terms using \eqref{xiv}.
This leaves us with the relation
\[
i \gamma_t = t^{-1} q(\xi_v,\xi_v,\xi_v) \gamma |\gamma|^2 + O(t^{-2}).
\]
Since $t^{-1}$ is not integrable at infinity, we see that 
it is not possible for the function $\gamma$ to have a nontrivial limit at infinity, which justifies our earlier claim that an asymptotic behavior as in \eqref{wish} cannot hold for the nonlinear evolution.

However, all is not lost. We can ensure that the last relation is satisfied if we allow a very mild dependence of $\gamma$ on $t$, 
precisely if we set $\gamma$ to satisfy the \emph{asymptotic equation}
\begin{equation}\label{asymptotic-t}
 i \gamma_t = t^{-1} q(\xi_v,\xi_v,\xi_v) \gamma |\gamma|^2   .
\end{equation}
This is an ode which only has global solutions provided that 
\[
\Im q(\xi,\xi,\xi) \leq 0, \qquad \xi \in \R .
\]
The case when $\Im q(\xi,\xi,\xi) < 0$ corresponds to a damping nonlinearity, and solutions for the asymptotic equation which decay to $0$. The more interesting case, which we will refer to as the conservative case, is when $q$ is real on the diagonal. In this case,
the solutions to the asymptotic equation \eqref{asymptotic-t}
have constant amplitude. 

In all cases, we remark that the asymptotic equation can be converted into an autonomous evolution with an exponential substitution, 
$t = e^s$. Then \eqref{asymptotic-t} becomes
\begin{equation}\label{asymptotic-s}
 i \gamma_s =  q(\xi_v,\xi_v,\xi_v) \gamma |\gamma|^2.   
\end{equation}

Hence, the objective of the analysis becomes to show that 
the solutions to \eqref{main} with small and localized data 
have the asymptotic behavior
\begin{equation}\label{u-asympt}
 u(t,x) \approx \frac{1}{\sqrt{t}} \gamma(\ln t,v) e^{i t \psi(v)},    
\end{equation}
where $\gamma$ solves the asymptotic equation \eqref{asymptotic-s}.
This has solutions of the form
\begin{equation}
\gamma(s,v) = e^{-is q(\xi_v,\xi_v,\xi_v)|\gamma_0(v)|^2} \gamma_0(v),   \end{equation}
depending on a function $\gamma_0$, which we will call 
the \emph{scattering profile} for the solution $u$. We will refer to such an asymptotic behavior as \emph{modified scattering}.

We remark that in this case we cannot expect $\gamma$ to be uniformly regular in $v$ as $s \to \infty$. However, this is harmless from the 
perspective of any asymptotic computation as above, as it only yields extra $\log t$ factors.

To summarize, we conclude that the objective of any asymptotic
analysis for the equation \eqref{main} is two-fold:

\begin{enumerate}[label=\alph*)]
\item Make a good choice for the profile $\gamma$, so that \eqref{u-asympt} holds.

\item Show that $\gamma$ approximately solves the asymptotic equation
\eqref{asymptotic-s}.
\end{enumerate}

The goal of these notes is to describe the method of \emph{testing by wave packets}, introduced by the authors in the context of a model NLS  problem in \cite{NLS} and then applied to quasilinear water wave evolutions in \cite{IT-global}, \cite{IT-t}. This method is applied here
in combination with energy estimates, which also raise some interesting questions due to the generality of the model considered, i.e. without any direct conservation laws.

\subsection{Asymptotic equations in the NLS context}
To set the stage for the presentation of our method, we will 
begin by first describing several alternative ideas which were 
proposed over the years in the context of the cubic NLS problem in one space dimension,
\begin{equation}\label{NLS}
i u_t - \frac12 \partial_x^2 u = \pm u |u|^2.    
\end{equation}
There one may take $a(\xi)= \frac12\xi^2$, in which case $\phi(v) = \frac12 v^2$ and $\xi_v=v$.

\begin{itemize}
\item[A.] Asymptotic equation in the Fourier space,
introduced by Hayashi-Naumkin~\cite{HN}, and refined by Kato-Pusateri~\cite{KP}. This is based on the idea that, taking 
a Fourier transform in an asymptotic formula like \eqref{u-asympt},
one obtains a related asymptotic for $\hat u$,
namely 
\[
\hat u(t,\xi) \approx \gamma(t,\xi) e^{-\frac{i}{2} t \xi^2}.
\]
Defining $\gamma$ by
\[
\gamma(t,\xi) = e^{\frac{i}{2} t \xi^2}u(t,\xi),
\]
one then seeks an asymptotic equation for the Fourier transform of the solutions,
\[
\frac{d}{dt} \hat u(t,\xi) = -i\xi^2 \hat u(t,\xi) + 
\pm  it^{-1}\, \hat u(t,\xi) | \hat{u}(t,\xi)|^2 + O_{L^{\infty}}(t^{-1-\epsilon}),
\]
where the first, respectively the second term on the right correspond to the linear, respectively the nonlinear part of the equation \eqref{main}.

\bigskip
\item[B.] Asymptotic equation in the physical space, introduced by Lindblad-Soffer~\cite{LS}; here the goal is to derive an asymptotic equation in the physical space along rays,
\[
(t \partial_t + x \partial_x+\frac12)   u(t,x) = 
\pm i t  u(t,x) | u(t,x)|^2 + O_{L^\infty}(t^{-\epsilon}),
\]
where the left hand side represents the linear contribution, while the right hand side represents the nonlinear contribution.

\bigskip
\item[C.] Nonlinear Fourier methods, developed by Deift-Zhou~\cite{MR1207209}, who used complete integrability and inverse scattering to obtain long range asymptotics via the steepest descent method. Unfortunately, these 
ideas are only available in the integrable case.

\bigskip

 \item[D.]  The {\em wave packet testing} method, introduced by the authors in \cite{NLS,IT-global,IT-t}, starts from the observation 
 that the methods described in A. and B. above lack balance
 when it comes to estimating the errors in the asymptotic equation.
 Working on the Fourier side, there are no linear errors but the nonlinear errors are large. On the physical side, there are no nonlinear errors,  but instead the linear errors are large. This led to the idea of looking 
 for a balanced way of defining the asymptotic profile $\gamma$, where  the linear and nonlinear errors are smaller and comparable. This is achieved by  testing the 
NLS solution with an approximate wave packet type linear wave $\uu_v$,
\[
\gamma(t,v) = \langle u, \uu_v \rangle_{L^2},
\]
where $\uu_v$ is both spatially localized in a $t^\frac12$ neighbourhood of the ray $x = vt$, and frequency localized in a dual $t^{-\frac12}$
neighbourhood of the frequency $\xi_v = v$. This perfectly balances
the linear and the nonlinear errors, and leads to results which 
are near optimal with respect to the regularity and decay of the initial data.
\end{itemize}

\section{ Global solutions for small localized data: the model case}
\label{s:compact}

In order to avoid distracting technicalities, in our first result we will make the simplifying assumption 
\begin{enumerate}
\item[(H5)] \emph{Frequency localized nonlinearity.} The symbol $q$ is compactly supported.
\end{enumerate}
This assumption makes the behavior of $a$ at infinity irrelevant.
Using the operator $L$, we define the following weighted time dependent
function space $X$:
\[
\| u(t)\|_{X}^2 := \|u(t)\|_{L^2} + \| Lu(t) \|_{L^2}^2.
\]
This will be used both for the initial data, and in order to measure the solution as it evolves in time. In particular, at time $t=0$ the $X$ norm measures the localization of  the initial data $u_0$,
\[
\|u_0\|_{X} \approx \| (1+x^2)^\frac12 u_0 \|_{L^2}.
\]
With these notations, our main result is as follows:

\begin{theorem}\label{t:comp}
  Assume that the conditions (H1-5) above are satisfied,
and that the  initial data for our equation \eqref{main}
satisfies:
\begin{equation}\label{small}
\| u(0)\|_{X} \lesssim \epsilon \ll 1.
\end{equation}
Then the solution exists globally  in time, with 
energy bounds 
\begin{equation}\label{energy}
\| u(t)\|_{X}  \lesssim \epsilon t^{C \epsilon^2},
\end{equation}
and pointwise decay 
\begin{equation}\label{disp-decay}
\| u(t)\|_{L^\infty}  \lesssim \frac{\epsilon}{\sqrt t}.
\end{equation} 
\end{theorem}

The rest of this section contains the proof of this result, organized as follows.
 Section~\ref{s:boot} provides the set-up for the main bootstrap argument. The energy estimates  leading to the bound \eqref{energy} are discussed in Subsection~\ref{s:energy};
this includes the energy bound in Proposition~\ref{p:energy} and the vector field bound in Proposition~\ref{p:Lu}.  Thus we arrive at the main objective of the paper, namely the wave packet analysis, which is considered in Subsection~\ref{s:wp}.

Finally, in Subsection~\ref{s:complete} we briefly discuss
the inverse problem, which is to reconstruct a solution given its asymptotic profile.

\subsection{Overview of the proof: A bootstrap argument.}
\label{s:boot}

The starting point of the proof is to make a bootstrap assumption for the pointwise bound,
\begin{equation} \label{boot}
\|u(t)\|_{L^\infty} \lesssim C \epsilon \la t\ra^{-\frac12}.
\end{equation}
Then the proof proceeds in two steps:
\medskip 

{\bf 1. Energy estimates:} 
Here the objective is to establish the energy bound
\begin{equation} 
\| u(t)\|_{X} \lesssim  \la t \ra^{C^2 \epsilon^2} \|u(0)\|_{X}.
\label{o:energy}
\end{equation}
This uses Gronwall's inequality in the equation for $u$, and then
a cubic correction to $Lu$. We note that, by the vector field bound 
in Proposition~\ref{p:vf}, this yields
\begin{equation}\label{o:first-point}
\|u(t)\|_{L^\infty} \lesssim t^{-\frac12} \|u(t)\|_X \lesssim  \epsilon \la t\ra^{C^2 \epsilon^2}.
\end{equation}
This step is carried out in Section~\ref{s:energy}.

\medskip
{\bf 2. Pointwise bounds:}
Here the goal is to improve the bootstrap assumption, and show that
\begin{equation}\label{o:point}
\| u(t)\|_{L^\infty} \lesssim \epsilon \la t \ra^{-\frac12} .
\end{equation}
 This step is carried out in Section~\ref{s:wp}, and  uses the method of {\em testing with wave-packets} to produce an asymptotic  profile $\gamma(t,v)$, which may be compared to the solution $u$ using the  bounds \eqref{o:energy}, respectively \eqref{o:first-point}.
Then it remains to prove an $\epsilon$ bound for $\gamma$, which is achieved by showing that $\gamma$ is a good approximate solution for  the asymptotic equation \eqref{asymptotic-t}.

\subsection{ Energy estimates}
\label{s:energy}
Our objective here is to prove energy estimates for $u$ and $Lu$, i.e. the bound \eqref{o:energy}.
In the case of $u$, we have the following straightforward result:

\begin{proposition}\label{p:energy}
Assume that $u \in L^2$ solves \eqref{main}.
Then 
\begin{equation}
\label{ee}
\frac{d}{dt} \|u\|_{L^2}^2 \lesssim \| u\|_{L^\infty}^2 \|u\|_{L^2}^2  . \end{equation}
\end{proposition}
We note that by Gronwall's inequality, this gives
\begin{equation} 
\| u(t)\|_{L^2} \lesssim  \epsilon \la t \ra^{C^2 \epsilon^2} 
\label{o:energy-first}
\end{equation}
which is the first half of \eqref{o:energy}.

\begin{proof} Multiplying equation \eqref{main} with $\bar{u}$, integrating over $x$, (H$1$) assumption, and adding the complex conjugate counterpart we obtain
\[
\frac{d}{dt}\|u\|_{L^2}^2= \int \left[ u_t\bar{u}+\bar{u}_tu \right]\,dx =-2\Im\int  Q(u, \bar{u}, u)\bar{u}\, dx.
\]
Thus, 
\[
\left| \frac{d}{dt}\|u\|_{L^2}^2\right| = \left| \int \left[ u_t\bar{u}+\bar{u}_tu \right]\,dx\right|  \lesssim  \Im \int \left| Q(u, \bar{u}, u)\bar{u} \right|\, dx\lesssim \| u\|_{L^\infty}^2 \|u\|_{L^2}^2 ,
\]
where for the last inequality we used the (H$3$) assumption on the nonlinearity $Q$ and pulled out the $L^{\infty}-$norm of $\vert u\vert^2$. Here one could 
think of the bound for $Q$ as a trilinear product bound,
using for instance the idea of separation of variables
discussed in a more general setting in Section~\ref{s:dyadic}.

\end{proof}

The more delicate matter is the energy estimate for $Lu$,
which solves the equation
\begin{equation}
(i\partial_t - A(D)) Lu = L Q(u,\bar u,u)  .  
\end{equation}
The difficulty is that the source term on the right does not 
directly satisfy a perturbative bound, e,g, of the form
\begin{equation}\label{Lu-not}
\| L Q(u,\bar u,u)  \|_{L^2} \lesssim \|u\|_{L^\infty}^2 (\|Lu\|_{L^2}+\|u\|_{L^2}).
\end{equation}
To address this issue, we will add a nonlinear correction 
to $Lu$, precisely
\[
L^{NL} u = Lu + t C(u,\bar u,u),
\]
where $C$  is  a well chosen trilinear form which has a smooth compactly supported symbol.  Precisely, we have the following:

\begin{proposition}\label{p:Lu}
There exists a smooth, compactly supported symbol $c(\cdot, \cdot, \cdot)$ with the 
property that the following estimate holds for solutions 
$u$ to \eqref{main}:
\begin{equation}\label{dt-Lu}
\frac{d}{dt} \|L^{NL} u\|_{L^2} \lesssim \| u\|_{L^\infty}^2 \|L^{NL} u\|_{L^2} + \| u\|_{L^\infty}^2 \|u\|_{L^2}
+ t \| u\|_{L^\infty}^4 \|u\|_{L^2}.
\end{equation}
\end{proposition}
In essence, the correction $C$ will be chosen so that a modified version of \eqref{Lu-not} holds; precisely,
that is the bound \eqref{Lu-not+} in the proof below.

 Given our bootstrap assumption \eqref{boot} and the $L^2$ estimate \eqref{o:energy-first} for $u$, Gronwall's inequality allows us to close the energy estimate for $L^{NL} u$  
 and obtain
 \begin{equation} 
\| L^{NL} u(t)\|_{L^2} \lesssim  \epsilon \la t \ra^{C^2 \epsilon^2}.
\label{o:energy-second}
\end{equation}
 Here we can use $L^{NL}$ and $L$ interchangeably since 
 by \eqref{boot} and the $L^2$ estimate \eqref{o:energy-first}
 we have a good bound for the difference,
 \begin{equation}\label{C-L2}
 \| t C(u,\bar u, u)\|_{L^2} \lesssim 
 t \|u\|_{L^\infty}^2 \|u\|_{L^2} \lesssim
 \la t \ra^{C^2 \epsilon^2}.
 \end{equation}
 Hence, the second part of \eqref{o:energy} also follows.

\begin{proof}
We write the equation for $L^{NL} u$ in the form
\[
(i\partial_t - A(D)) L^{NL}u = L Q(u,\bar u,u) + i C(u,\bar u,u) 
- t  R_3(u,\bar u,u) - t R_5(u,\bar u, u, \bar u,u) := f,
\]
where $R_3$ and $R_5$ are translation invariant multilinear forms with smooth compactly supported symbols, $r_3$, and $r_5$ respectively. Furthermore, the symbol of $R_3$ 
is given by
\[
r_3(\xi_1,\xi_2,\xi_3) : = c(\xi_1,\xi_2,\xi_3)(a(\xi_1)-a(\xi_2) + a(\xi_3)-a(\xi_1-\xi_2+\xi_3)).
\]

To prove \eqref{dt-Lu} it suffices to have the following bound for the above source term $f$:
\begin{equation}\label{f-L2}
\|f\|_{L^2} \lesssim      \| u\|_{L^\infty}^2 \|L u\|_{L^2} + \| u\|_{L^\infty}^2 \|u\|_{L^2}
+ t \| u\|_{L^\infty}^4 \|u\|_{L^2}.
\end{equation}
Here we used \eqref{C-L2} to replace $L^{NL} u$ by $Lu$
in the right. 

The terms $i C(u,\bar u,u)$ respectively  
$t R_5(u,\bar u, u, \bar u,u)$ can be directly estimated by the second, respectively the third term on the right 
in \eqref{f-L2}, without using any structural properties 
of the corresponding symbols. Hence it remains to consider the expression
\[
L Q(u,\bar u,u) - t  R_3(u,\bar u,u).
\]
Our objective will be to choose the trilinear form $C$ with the property  that the bound \eqref{Lu-not} holds for 
the above expression. The choice of the symbol $c$ is given by the following 
algebraic division Lemma:

\begin{lemma}\label{l:division}
There exist smooth, compactly supported symbols $c$, $c_1$, $c_2$ and $c_3$ so that the following algebraic relation holds:
\begin{equation}\label{division}
\ell(x,\xi) q(\xi_1,\xi_2,\xi_3) - tc(\xi_1,\xi_2,\xi_3) (a(\xi_1)-a(\xi_2)+a(\xi_3)-a(\xi)) =  \sum_{j=1}^3 c_j(\xi_1,\xi_2,\xi_3) \ell(x,\xi_j)
\end{equation}
whenever $\xi = \xi_1-\xi_2+\xi_3$.
\end{lemma}
We first use Lemma~\ref{l:division} to complete the proof of Proposition~\ref{p:Lu}.
The relation \eqref{division} translates into 
the following operator identity:
\[
L Q(u,\bar u,u) - t  R_3(u,\bar u,u) =
C_1(Lu,\bar u,u) + C_2(u, \overline{Lu},u) + C_3(u,\bar u, Lu)
+  D(u,\bar u, u),
\]
where $D$ has symbol
\[
d(\xi_i,\xi_2,\xi_3) = i(\partial_{\xi_1} c_1 - \partial_{\xi_2} c_2 + \partial_{\xi_3} c_3).
\]
This directly implies the bound
\begin{equation}\label{Lu-not+}
\| L Q(u,\bar u,u) - t  R_3(u,\bar u,u)\|_{L^2} \lesssim 
\|u\|_{L^\infty}^2 (\|Lu\|_{L^2} + \|u\|_{L^2}),
\end{equation}
and the proof of Proposition~\ref{p:Lu} is concluded.

\end{proof}

It remains to prove Lemma~\ref{l:division}.

\begin{proof}[Proof of Lemma~\ref{l:division}]
We start with some simplifications. Without any restriction in generality we can set $q=1$, with the minor proviso that now we discard
the compact support property for $c$ and $c_j$. Secondly, we can separate the $x$ and the $t$ component of the above identity, and conclude that we need to satisfy two identities:
\[
c_1-c_2+c_3 = 1,
\]
and
\[
c_1 a_\xi(\xi_1) - c_2 a_\xi(\xi_2) + c_3 a_\xi(\xi_3) = a_\xi(\xi) + c(\xi_1,\xi_2,\xi_3)(a(\xi_1)-a(\xi_2) + a(\xi_3) - a(\xi)).
\]
Simplifying further, we set $c_j=1$ so that the first identity is trivially satisfied. We are left with 
\[
c(\xi_1,\xi_2,\xi_3) = \frac{a_\xi(\xi_1) - a_\xi(\xi_2) + a_\xi(\xi_3)
- a_\xi(\xi)}{ a(\xi_1)-a(\xi_2) + a(\xi_3) - a(\xi)},
\]
where we need to show that the quotient is smooth.

Since $a$ is strictly convex (or concave), it is easily seen that 
the denominator can only vanish on the set 
\[
 \mathcal{D} = \{ \xi_1+\xi_3 = \xi_2 +\xi\} 
\]
if and only if
\[
\{ \xi_1,\xi_3\} = \{ \xi,\xi_2 \}.
\]
We claim that the denominator admits a representation of the form
\begin{equation}\label{d4a}
a(\xi_1)-a(\xi_2) + a(\xi_3) - a(\xi) = (\xi-\xi_1)(\xi-\xi_3)
b(\xi_1,\xi_2,\xi_3)
\end{equation}
with $b$ smooth and nonzero.

We start with the standard representation
\[
a(\xi_1) - a(\xi) = (\xi-\xi_1)a_1(\xi_1,\xi)
\]
with smooth, symmetric  $a_1$, and then 
write on $\mathcal{D}$
\[
a(\xi_1)-a(\xi_2) + a(\xi_3) - a(\xi) = (\xi-\xi_1)(a_1(\xi_1,\xi) - 
a_1(\xi_3,\xi_2)).
\]
Then we repeat the process for $a_1$ to pull out a factor of 
\[
\xi - \xi_3 = \xi_1-\xi_2.
\]
This yields a representation as in \eqref{d4a}, with a smooth $b$.
It remains to verify that $b$ is nonzero, for which we compute 
$b$ on the zero set. Suppose for instance that $\xi = \xi_1$
and $\xi_2 = \xi_3$. Then $a_1(\xi,\xi_1)= a'(\xi_1)$, and 
\[
b = \frac{a'(\xi) - a'(\xi_3)}{\xi - \xi_3},
\]
which is nonzero due to the strict convexity (concavity) of $a$.
We also remark that at the double zero, when all frequencies are equal, we have 
\[
b= a''(\xi) \neq 0.
\]

Next we consider the numerator, for which we also have a representation
\begin{equation}\label{d4da}
a_\xi(\xi_1)-a_\xi(\xi_2) + a_\xi(\xi_3) - a_\xi(\xi) = (\xi-\xi_1)(\xi-\xi_3)
b_1(\xi_1,\xi_2,\xi_3).
\end{equation}
Here $b_1$ is again smooth, but not necessarily nonzero.

Finally, we divide the expressions in \eqref{d4a} and \eqref{d4da}
to obtain
\[
c(\xi_1,\xi_2,\xi_3)=\frac{b_1(\xi_1,\xi_2,\xi_3)}{b(\xi_1,\xi_2,\xi_3)},
\]
which is easily seen to be smooth as the denominator is nonzero.

This concludes the proof of the lemma.
\end{proof}

\subsection{Wave packet testing}
\label{s:wp}

Our objective here is to describe the method of testing by wave packets, and show how it applies to our problem in order to define the asymptotic profile $\gamma$ and to show that is approximately solves the asymptotic equation, with the final objective of establishing the uniform bound for $\gamma$, which in turn implies the uniform bound for the solution $u$.

We will begin with a short description of wave packes on a fixed spatial scale, both for the linear and then for the nonlinear model.
Then we discuss the wave packets on a time dependent scale, which are critical in our analysis here. Finally, we use these wave packets 
to construct the asymptotic profile $\gamma$, and prove that it has the desired properties.

\subsubsection{Linear wave packets on a fixed scale}
The idea here is to look for solutions to the linear equation \eqref{linear} which are localized near a trajectory for the Hamilton flow,
\[
(x_0,\xi_0) \to (x_0+t a_\xi(\xi_0), \xi_0).
\]
Most desirably, this localization should occur both in position and in frequency, on the sharp, uncertainty principle scale.
The localization scales are denoted as follows:
\[
(\delta x, \delta \xi), \qquad  \delta x \cdot  \delta \xi \approx 1 \ \ \text{(uncertainty principle)}
\]
The first step is to choose these scales so that
this localization is coherent up to a given time $T$.
Heuristically, the varying group velocities within the $\delta \xi$
range leads to position variations for the Hamilton flow
up to the time $T$, which are given by
\[
\delta x = T a_{\xi\xi}(\xi_0) \delta \xi.
\]
Matching this with the uncertainty principle relation, we obtain
the localization scales adapted to the time scale $T$, namely 
\[
\delta x = T^\frac12 a_{\xi\xi}(\xi_0)^{\frac12}, \qquad
\ \ \delta \xi = T^{-\frac12} a_{\xi\xi}^{-\frac12}.
\]

So far we have only looked at the coherence at the level of the Hamilton flow. Next, we ask whether one can realize this localization at the level of actual solutions. This leads to the so called
wave packet solutions, which are approximately of the form
\[
u(x,t) \approx \gamma \  \chi( (\delta x)^{-1}(x - x_0 - t a_\xi(\xi_0))) e^{i (x \xi_0 - t a(\xi_0))} .
\]
Here one can adopt two equally useful view points. On one hand, keeping the Schwartz function $\chi$ independent of $t$, one obtains an approximate solution to \eqref{linear}, with errors which are small\footnote{Say in $L_t^1 L_x^2$.} up to time $T$. On the other hand, one can start with a given Schwartz function $\chi$ at $t=0$, and show that the representation above persists exactly with a time dependent Schwartz
function $\chi$ which satisfies uniform bounds up to time $T$.
This can be achieved via Fourier analysis, but also via
energy estimates, using the operator $L$ defined above, as well as its powers. This philosophy applies as well in variable coefficient case, see e.g. \cite{KT}.

One can think of general $L^2$ solutions to the linear flow \eqref{linear} as linear, square summable superpositions of wave packets, which can be taken either relative to a discrete
set of centers $(x,\xi)$ (wave packet parametrices) or with respect 
to a continuous set of centers, akin to phase space transform methods\footnote{a.k.a. the Bargman or the FBI transform, see e.g. \cite{Pisa}.}

Finally, we remind the reader that, under the name of Knapp counterexamples, wave packets have been used to show that  
Strichartz and the dispersive estimates are sharp.

\subsubsection{Nonlinear wave packet solutions on a fixed scale}
Here we switch to the nonlinear flow \eqref{main}, and consider 
wave packet solutions, which are localized on scales similar to the ones above. The new factor here is the amplitude of the nonlinearity,
which we denote by $\mathcal{M}$. Then the linear ansatz for wave packets
is modified to
\[
u(x,t) \approx \gamma(t, (\delta x)^{-1}(x - x_0 - t a_\xi(\xi_0))) e^{i (x \xi_0 - t a(\xi_0))} ,
\]
where the modulation factor $\gamma$ is taken to have size $\mathcal{M}$.

For functions $u$ with wave packet localization near frequency $\xi_0$, it turns out that 
the nonlinearity is well approximated by 
\[
Q(u,\bar u, u) = q(\xi_0,\xi_0,\xi_0)  |u|^2 u + O(\delta \xi |\mathcal{M}|^3) ,
\]
where the error can be thought off as perturbative provided that the 
amplitude is small enough,
\[
\delta \xi \mathcal{M}^2 \ll 1.
\]
Assuming this is the case, the amplitude function $\gamma$
should approximatively solve the asymptotic ode
\[
i \dot \gamma = q(\xi_0,\xi_0,\xi_0) \gamma |\gamma|^2.
\]
This in turn is conservative if $q(\xi_0,\xi_0,\xi_0) $ is real.

One should relate these heuristics with the idea of NLS approximation,
which roughly asserts that solutions with this type of localization 
and amplitude are well approximated by solutions to a suitable NLS problem, obtained by replacing the symbol $a$ with its quadratic approximation at $\xi_0$, and the cubic form $Q$ by
$q(\xi_0,\xi_0,\xi_0)  |u|^2 u$.  For more information on this  
we refer the reader to \cite{NLS-approx}, \cite{MR3409892} and further references therein.

\subsubsection{Linear wave packets with time dependent scale}
Working with packets with fixed scales is useful for the study 
of the local problem, but not so much for the global in time evolution.
Because of this, we will now consider global in time approximate wave packet solutions for the linear problem \eqref{linear}. To understand their structure, we recall that the spatial scales associated to time scale $t$ at velocity $v$ and associated frequency $a'(\xi_v) = v$
are given by
\[
\delta x = t^\frac12 [a_{\xi\xi}(\xi_v)]^{\frac12}, 
\ \ \delta \xi = t^{-\frac12} [a_{\xi\xi}(\xi_v)]^{-\frac12}.
\]
We now replicate the previous wave packet ansatz, 
but do it globally in time, with a time dependent scale.
Thus we define the linear wave packet $\uu_v$ associated with velocity $v$ by
\begin{equation}\label{wp-def}
{\bf u}_v = a''(\xi_v)^{-\frac12}
 \chi(y) e^{it \phi(x/t)}, \qquad y = \frac{x-vt}{t^\frac12 a''(\xi_v)^{\frac12}}.
\end{equation}
where $\chi$ is a compactly supported smooth function, which we normalize 
so that 
\[
\int \chi(y) dy = 1.
\]
This is a good approximate solution for the linear flow on dyadic time scales:
\begin{equation}\label{uu-approx}
(i \partial_t - A(D)) {\uu}_v \approx  O( t^{-1})\uu.
\end{equation}

\begin{figure}
\begin{tikzpicture} \begin{axis}
[yticklabels={,,}, xticklabels={,,},  axis y line=center, axis x line=middle, xlabel=$x$,ylabel=$t$ ] 
\addplot[smooth,magenta,mark=none, domain=0:2] {x}; 
\addplot[smooth,blue,mark=none, domain=0:1.5] {1.5*x-x^2/8}; 
\addplot[smooth,blue,mark=none, domain=0:2.5] {.5*x+x^2/8}; 
\addplot[smooth,black,mark=none, domain=.5:0] {0}; 
\addplot[smooth,black,mark=, domain=1.18:2.1] {1.6};
\node at (axis cs:1.3,0.2) {$\delta x = \sqrt{t}, \ \ \delta \xi = \dfrac{1}{\sqrt{t}}$ };
\node[rotate=45] at (axis cs:1.1,1.2) {$x=vt$};
\node at (axis cs:1.5,1.7) {$ \sqrt{t}$};
\end{axis} 
\end{tikzpicture}
\caption{The support of a wave packet with velocity $v$.}
\end{figure}
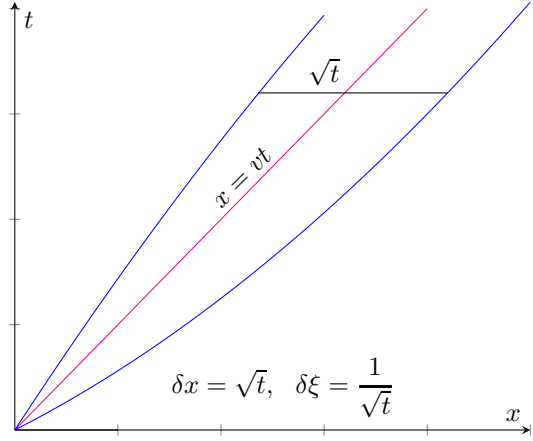
However, we carefully remark that this is not a good approximate solution globally in time. Indeed, any global solution should disperse,
rather than stay concentrated near a ray. As a corollary of this remark, we note that the above relation will still remain satisfied if we replace $\uu_v$, say, by $t^\mu \uu_v$. The choice we made above 
is for convenience only, and not at all intrinsic. If one wanted for instance to have solutions which stay bounded in $L^2$, then choosing 
$\mu = -\frac14$ would be the appropriate choice.

However, there is one advantage for our normalization,
which is seen when one attempts to gain a better understanding 
of the error term in the linear equation for $\uu_v$. 
Precisely we have the following:

\begin{lemma}\label{l:Puu}
The wave packet $\uu_v$ defined above solves a linear 
equation of the form
\begin{equation}\label{wp-eqn}
 (i \partial_t - A(D)) {\bf u}_v =  t^{-\frac32} L \uu^I_v +  t^{-\frac32} \rr_v,
\end{equation}
where $\uu^I_v$ and $\rr_v$ have a wave packet form similar to $\uu_v$.

\end{lemma}
Compared to \eqref{uu-approx}, the lemma provides a more accurate description of the $t^{-1}$ term. The function $\uu^I_v$ here is quite explicit,
\[
\uu^I_v = - \frac{i}{2a''} (\chi'(y) - i y \chi(y)) e^{it \phi(x/t)} .
\]
This is not important later on, what matters is that the operator $L$
is applied to it. The function $\rr_v$ is less explicit but this is also not important as $\rr_v$ only plays a perturbative role later on.

\begin{proof}
While not absolutely necessary, here it is helpful to simplify the 
problem using some simple linear transformations:
\begin{itemize}
    \item Using a Galilean transformation $x-vt \to x$, the problem
    reduces to the case $v=0$. Note that this changes $a$ by a linear 
    term.
    \item Using a spatial phase shift, $\uu_v$ to $\uu_v e^{-ix \xi_v}$,
we can also ensure that $\xi_v = 0$. This translates $a$ by $\xi_v$.
\item Using a temporal phase shift, the problem reduces also to the case $a(0)=0$.
\item If $a''(0) < 0$, we can shift to $a''(0) > 0$ by replacing 
$u$ with $\bar u$ (and thus $a(\xi)$ by $a(-\xi)$).
\end{itemize}
After these simplifications, we are now in the case when
\[
v = 0, \quad \xi_v = 0,\quad  a(0) = 0, \quad a'(0)=0. 
\]
This in turn implies that 
\[
\phi(0) = 0, \quad \phi'(0) = 0, \quad \phi''(0) = \frac1{a''(0)}.
\]

Now we finally compute the equation for $\uu_0$.
For this, we use the Taylor expansion of $a$ at $0$,
\[
a(\xi) = \frac12 a''(0) \xi^2 + O(\xi^3)  .
\]
The contribution of the $\xi^3$ sized error has size $t^{-\frac32}$,
and may be included into $\rr_0$. Similarly, we have 
\[
\ell(x,\xi) = x- ta'(\xi) = x- t \xi a''(0) + O(t\xi^2),
\]
where the contribution of the $t \xi^2$ tail also 
can be included into $\rr_0$.

Since the $v$ in the lemma was set to $0$, in what follows we use
the notation $v= x/t$. Hence we have
 \[
 \begin{aligned}
 (i \partial_t - A(D)) {\bf u}_0 = & \ (i \partial_t + \frac12 a''(0) \partial_x^2 ) {\bf u}_0 + t^{-\frac32} \rr_0
\\
= & \ - \frac{i}{2} t^{-\frac32} x  (a''(0))^{-1}\chi'(y) e^{it \phi(v)}
- (a''(0)^{-\frac12}( \phi(v) - v \phi'(v)) \chi(y) e^{it\phi(v)}
\\ & + \frac12 (a''(0))^{-\frac12}\left( \frac{ \chi''(y)}{t}
+ 2 i  (a'')^\frac12 \frac{ \chi'(y) \phi'(v)}{t^\frac12 } - a''(0)\chi(y) \phi'^2(v)
+ i a''(0) \frac{\chi(y) \phi''(v)}{t}
\right) e^{it\phi}+ t^{-\frac32} \rr_0.
\end{aligned}
\]
Noting the leading order cancellation
\[
\phi(v) - v \phi'(v) = -\frac12 a''(0)  \phi'^2(v) + O(v^3),
\]
where the last term only contributes to the error,
we obtain 
\[
 \begin{aligned}
 (i \partial_t - A(D)) {\bf u}_0 = & \ 
- \frac{i}{2} t^{-\frac32} ( x + it a''(0) \partial_x)[a''(0)^{-1}\chi'(y) e^{it \phi(v)}] 
+\frac{i}{2} a''(0)^\frac12 \partial_x(\chi(y) \phi'(v))
e^{it \phi(v)} + t^{-\frac32} \rr_0
\\
= & \ 
- \frac{i}{2} t^{-\frac32} ( x + it a''(0) \partial_x)[a''(0)^{-1}\chi'(y) e^{it \phi(v)}] 
+\frac{i}{2} a''(0)^\frac12 ( \partial_x - i \phi'(v)) [\chi(y) \phi'(v)
e^{it \phi(v)}] + t^{-\frac32} \rr_0.
  \end{aligned}
\]
Since $\phi'(v) = (a''(0))^{-1}v + O(v^2) $, we can rewrite 
the second term on the right to get
\[
\begin{aligned}
 (i \partial_t - A(D)) {\bf u}_0
 = & \ - \frac{i}{2} a''(0)^{-1} t^{-\frac32}
 ( x + it a''(0) \partial_x)
 \left[ \chi'(y) e^{it \phi(v)} 
+i y \chi(y) 
e^{it \phi(v)}\right] + t^{-\frac32} \rr_0
\\
= &\ 
- \frac{i}{2} a''(0)^{-1}t^{-\frac32}  L
 \left[ (\chi'(y)+i y\chi(y)) e^{it \phi(v)} \right] + t^{-\frac32} \rr_0
\end{aligned}
\]
as needed.

\end{proof}

We also need to consider the $v$ dependence of $\uu_v$.

\begin{lemma}\label{l:dv-uv}
The wave packet $\uu_v$ defined above solves a linear 
equation of the form
\begin{equation}\label{dv-wp}
 \partial_v {\bf u}_v =    L \uu^{II}_v +   \rr_v,
\end{equation}
where $\uu^{II}_v$ and $\rr_v$ have a wave packet form similar to $\uu_v$.
\end{lemma}

Here we have 
\[
\uu_v^{II} = - i [a''(\xi_v)]^{-\frac32} \chi(y) e^{it \phi(x/t)}.
\]

\begin{proof}
Differentiating with respect to $v$ yields
\[
\partial_v {\bf u}_v = - t \partial_x [ \chi (y)] e^{i t \phi(x/t)}  + \rr_v ,
\]
and the first term on the right is similar to the second term on the right 
in the computation in the previous lemma.
\end{proof}

\subsubsection{Wave packet testing}
As described earlier, we will define our  asymptotic profile function $\gamma$
by
\begin{equation}\label{def-gamma}
\gamma(t,v) = \langle u, {\bf u}_v\rangle_{L^2}.
\end{equation}
Now our objective is two-fold:
\begin{itemize}
    \item To show that $\gamma$ provides a good approximation for $u$, in the sense of 
    \eqref{u-asympt}.
    \item To show that $\gamma$ is an approximate solution for the asymptotic equation
    \eqref{asymptotic-t}.
\end{itemize}

\subsection{ Bounds for $\gamma$}
\label{s:bd-gamma}
Here we establish some base-line bounds for $\gamma$, using the energy estimates 
in Proposition~\ref{p:energy}:

\begin{proposition}\label{p:gamma}
Assume that $u$ satisfies the energy bounds in  Proposition~\ref{p:energy}.
Then $\gamma$ satisfies
\begin{equation}
\| \gamma \|_{L_v^2} + \| \partial_v \gamma\|_{L^2_v} \lesssim \epsilon t^{C^2 \epsilon}.
\end{equation}
\end{proposition}

\begin{proof}
If we bound $\uu_v$ by 
\[
|\uu_v(t,x_1)| \lesssim  \frac{1}{1+ t |v-v_1|^2}, \qquad x_1 = t v_1, 
\]
then the $L^2$ bound for $\gamma$ can be interpreted as a convolution estimate,
as 
\[
|\gamma(v)| \lesssim t |u(vt)| \ast_v \frac{1}{1+ tv^2} ,
\]
where the convolution kernel is integrable. By Young's inequality this yields
\[
\| \gamma \|_{L^2_v}
\lesssim \sqrt{t} \|u(vt)\|_{L^2_v} 
\lesssim \|u\|_{L^2_x}
\]
as needed.

For the $L^2$ bound for $\partial_v \gamma$ we first apply Lemma~\ref{l:dv-uv}.
Then we obtain the convolution bound
\[
|\partial_v \gamma(v)| \lesssim  t (|Lu|+|u|)(vt) \ast_v \frac{1}{1+tv^2} ,
\]
and then conclude as above.
\end{proof}

\subsubsection{ Approximate profile.}

Our goal here is to estimate the  difference
\[
r(t,x) = u(t,x) - \frac{1}{\sqrt{t}}\gamma(t,v) e^{i t \phi(v)}, \qquad v = \frac{x}{t}
\]
as follows:

\begin{proposition}\label{p:u-gamma}
Assume that $u$ satisfies the energy bounds in  Proposition~\ref{p:energy}.
Then the above error $r$ satisfies the uniform bound
\begin{equation}
\| r \|_{L^\infty} \lesssim \epsilon t^{-\frac{3}{4}+C^2 \epsilon},
\end{equation}
and the $L^2$ bound
\begin{equation}
\| r \|_{L^2} \lesssim \epsilon t^{-\frac{1}{2}+C^2 \epsilon}.
\end{equation}
\end{proposition}

\begin{proof}
We represent
\[
\sqrt{t} e^{-it\phi(\blue{v})} r(t,tv ) = \langle u, \ww_v \rangle ,
\]
where on the right we use the $L^2_x$ pairing, with  
\[
\ww_v = \sqrt{t} e^{it\phi} \delta_{x=vt}  - \uu_v.
\]
Using the normalization $\int \chi = 1$ we rewrite $\ww_v$ as 
\[
\ww_v = t^\frac12 \partial_x(\chi_1 (y) \sgn(y)) e^{it \phi} =  t^\frac12 (\partial_x-i\phi'(x/t)) 
[\chi_1 (y) \sgn(y) e^{it \phi}],
\]
where $\chi_1$ is 
\[
\chi_1(y) =\left\{ 
\begin{array}{lc}
 - \int_{-\infty}^y  \chi (z) dz & y < 0  \\ \\
- \int_y^\infty    \chi (z) dz & y > 0  ,
\end{array}
\right.
\]
which leads to 
\[
\sqrt{t} e^{-it\phi(v)} r(t,tv) = t^{-\frac12} \langle \tL u, \uu^{III}_v \rangle , \qquad \uu^{III}_v = \chi_1 (y) \sgn(y) e^{it \phi}.
\]
Now $\uu^{III}_v$ has the same size and localization as $\uu$,
so we can argue  as in the 
proof of Proposition~\ref{p:gamma}
that
\[
|r(t,vt)| \lesssim |\tL u (t,vt)| \ast_v
\frac{1}{1+tv^2}.
\]
Then by Young's inequality we conclude that
\[
\| r(t,x)\|_{L^2_x} = t^{\frac12} 
\| r(t,vt)\|_{L^2_v}
\lesssim  \| \tL u(t,vt)\|_{L^2_v} 
= t^{-\frac12} \| \tL u(t,x)\|_{L^2_x} 
\]
respectively 
\[
\| r\|_{L^\infty} \lesssim t^{-\frac14}  \| \tL u(t,vt)\|_{L^2_v}
= t^{-\frac34} \|\tL u\|_{L^2_x}.
\]
Now we can conclude using Lemma~\ref{l:tL}.
\end{proof}

\subsubsection{ The asymptotic equation for $\gamma$}

Here we prove the following:

\begin{proposition}\label{p:gamma-er}
Assume that $u$ satisfies the energy bounds 
\[
\|u\|_{X} \lesssim \epsilon \la t \ra^{C^2 \epsilon^2}.
\]
Then $\gamma$ solves the asymptotic equation
\begin{equation}\label{gamma-asympt}
    \dot \gamma(t,v) = i q(\xi_v,\xi_v,\xi_v) t^{-1} \gamma(t,v) |\gamma(t,v)|^2 
+ f(t,v),
\end{equation}
where $f$ satisfies the uniform bound
\begin{equation}
\label{e:f-inft}
\| f \|_{L^\infty} \lesssim \epsilon t^{-\frac{5}{4}+3C^2 \epsilon},
\end{equation}
and the $L^2$ bound
\begin{equation}
\label{e:f-l2}
\| f \|_{L^2_v} \lesssim \epsilon t^{-\frac32+3C^2 \epsilon}.
\end{equation}
\end{proposition}
\begin{proof}
We compute
\[
\dot\gamma(t,v) = -i \langle (i\partial_t - A(D)) u, \uu_v \rangle+
i  \langle u, (i\partial_t - A(D)) \uu_v \rangle := I_1(t,v)+I_2(t,v).
\]

For $I_2$ we use Lemma~\ref{l:Puu} to write
\[
I_2(t,v) = i t^{-\frac32} (\langle Lu, \uu^I_v \rangle +  \langle u, \rr_v \rangle).
\]
This allow us to bound its size both in $L^\infty$, using H\"older's inequality,
and in $L^2$ via convolution bounds.

The expression $I_1$, on the other hand, has the form
\[
I_1(t,v) = i \langle Q(u,\bar u, u), \uu_v \rangle.
\]
Here we first use 
the bounds for $r$ in 
Proposition~\ref{p:u-gamma} in order to substitute 
$u$ with $t^{-\frac12} \gamma(t,x/t) e^{it\phi(x/t)}$ modulo acceptable errors,
\[
I_1(t,v) = i t^{-\frac32} \gamma(t,v) |\gamma(t,v)|^2 
\langle Q( \gamma(t,x/t) e^{it \phi},
\bar \gamma(t,x/t)e^{-it\phi},\gamma(t,x/t) e^{it\phi}),\uu_v\rangle + f ,
\]
where the error $f$ is as in \eqref{e:f-inft}, \eqref{e:f-l2}.

Then we take advantage of the fact that the kernel of $Q$ is localized on the unit scale
in order to replace $\gamma(t,x/t)$ with $\gamma(t,v)$, again with acceptable errors, which are estimated using the bounds for $\partial_v \gamma$ in Proposition~\ref{p:gamma}.
Thus we get
\[
I_1(t,v) = i t^{-\frac32} \gamma(t,v) |\gamma(t,v)|^2 
\langle Q( e^{it \phi},e^{-it\phi},e^{it\phi}),\uu_v\rangle + f.
\]

Finally, a semiclassical computation shows that 
\[
Q( e^{it \phi},e^{-it\phi},e^{it\phi}) = q(\xi_v,\xi_v,\xi_v) e^{it \phi} + O(t^{-1}),
\]
so the desired asymptotic equation follows.

\end{proof}

\subsection{ Conclusion}

Here we show how to close the bootstrap argument, and prove that 
the global result follows as a consequence of the results in Propositions~\ref{p:u-gamma} 
and \ref{gamma-asympt}.

The bootstrap argument closes as follows:
\begin{itemize}
    \item  Proposition~\ref{p:Lu} gives the energy bounds on our corrected vector field $L^{NL}$, where we made use of the energy estimates on $u$ obtained in \eqref{o:energy-first} and of the bootstrap assumption expressed in \eqref{boot} to get
    \[
    \Vert L^{NL}u(t)\Vert_{L^2}\lesssim \epsilon \langle t\rangle^{C^2\epsilon^2} .
    \]
\item  We use the pointwise decay bound obtained in Proposition~\ref{p:u-gamma} for the difference between the asymptotic profile $\gamma (t,v)$ and solution $u$ to conclude that the error term $f$ in the asymptotic equation \eqref{gamma-asympt} for  $\gamma$ is acceptable,
i.e. has better than $t^{-1}$ decay, 
as stated in Proposition~\ref{p:gamma-er}. Integrating \eqref{gamma-asympt} leads to a pointwise bound for $\gamma$:
\[
\vert \gamma (t,v)\vert \leq \vert \gamma (1,v)\vert +\int _1^t \vert f(s,v)\vert \, ds.
\]
Here, we use the energy bound \eqref{o:energy-first} at time $t=1$ and the pointwise bound on $\gamma$ to conclude that
\[
\Vert \gamma (1,v)\Vert_{L^{\infty}}\lesssim \epsilon,
\]
as well as the pointwise bound on $f$ given in Proposition~\ref{p:gamma-er}, and get
\[
\vert \gamma (t,v)\vert \leq \vert \gamma (1,v)\vert +\int _1^t \vert f(s,v)\vert \, ds\lesssim \epsilon . 
\]
\item Lastly, from the above estimate and Proposition~\ref{p:u-gamma}, it follows that the  pointwise bound on $u$ is
\[
\left| u\right| \lesssim \epsilon t^{-\frac12},
\]
which, under the constraint $1\ll C$, concludes the  bootstrap argument. 
   
  \end{itemize}

\subsection{Modified scattering and asymptotic completeness}
\label{s:complete}

An immediate consequence of the approximate asymptotic equation \eqref{gamma-asympt} for $\gamma$ is that, as $t \to \infty$, the function $\gamma$ converges to a solution 
$\tgamma$ to the exact asymptotic equation,
\[
\dot \tgamma = i q(\xi_v,\xi_v,\xi_v) t^{-1} \tgamma(t,v) |\tgamma(t,v)|^2 ,
\]
which can be represented in the form
\[
\tgamma(t,v) = W(v) e^{i q(\xi_v,\xi_v,\xi_v) \ln t |W(v)|^2}. 
\]

We will refer to the function $W$ as the asymptotic profile of the solution $u$, which is now asymptotically described as
\[
u(t,x) \approx \frac{1}{\sqrt{t}}W(v) e^{i q(\xi_v,\xi_v,\xi_v) \ln t |W(v)|^2} e^{it \phi(x/t)}.
\]
Then it is natural to consider the relation between the initial data $u_0$ and the asymptotic profile $W$, via bounds for the difference
\begin{equation}\label{scat-error}
e(t,x)=    u(t,x) - \frac{1}{\sqrt{t}}W(v) e^{i q(\xi_v,\xi_v,\xi_v) \ln t |W(v)|^2} e^{it \phi(x/t)}.
\end{equation}
In order to avoid  any discussion 
of the asymptotic behavior of $a$ at infinity, here we choose some compact frequency interval $I$ so that the 
the symbol $q$ of the nonlinearity is supported in $I^3$, and assume that the initial data $u_0$ is frequency localized  in $I$. Then the associated range of velocities is $J = a'(I)$.

\begin{theorem} \label{t:scat-model}
a) For each initial data $u_0$ satisfying the smallness condition \eqref{small} and which is frequency localized in $I$, there exists an asymptotic profile $W \in H^{1-C\epsilon}$, supported in $J = a'(I)$ and
with the property that 
\begin{equation}
\| W\|_{H_v^{1-C\epsilon}}    \lesssim \epsilon,
\end{equation}
for which the above difference satisfies the $L^2$ bounds
\begin{equation}
\| e \|_{L_x^2} \lesssim \epsilon t^{-\frac12+C\epsilon},
\end{equation}
as well as the $L^\infty$ bounds
\begin{equation}
\| e \|_{L_x^\infty} \lesssim \epsilon t^{-\frac34+C\epsilon}.
\end{equation}
Furthermore, the map $u_0 \to W$ is injective.

b) For each $W$  supported in $J$ and satisfying 
\begin{equation}
\| W\|_{H^{1+C\epsilon}_v}    \lesssim \epsilon,
\end{equation}
there exists an associated initial data $u_0$ satisfying the smallness condition \eqref{small} and frequency localized in $I$ so that $W$ is the asymptotic profile of $u_0$.
\end{theorem}

Often one refers to the first property as the scattering property (modified scattering)
and the second as the existence of wave operators 
(modified wave operators in our context). Together, they are called 
the \emph{asymptotic completeness} property.

We also remark on the slight imperfection in the above result, connected with the $\pm C\epsilon$ terms in the Sobolev indices. These are largely unavoidable due to the $\log t$
terms in the phase, though one might possibly replace small powers with logs. 

\begin{proof}
The argument here repeats the one in \cite{NLS}, and is omitted.
\end{proof}

\section{Global solutions for small localized data: the general case}
\label{s:general}

Here we consider several possible extensions of our main result, where  we drop the compact support assumption on the symbol of the nonlinearity $Q$. Then we can no longer work with frequency localized data, so instead we will have to assume a suitable 
Sobolev type regularity at infinity. Precisely, we will 
define the space $X$ by 
\begin{equation}\label{def:X}
\| u \|_{X}^2 = \| \Lambda_0(D) u \|_{L^2} + \|\Lambda_1 (D) L u\|_{L^2} 
\end{equation}
with suitable multiplier weights $\Lambda_0$ and $\Lambda_1$.
The question we ask is 

\begin{question}
Given the symbols $a,q$ and the above space $X$, under what assumptions does a small initial data in $X$ guarantee global solutions and modified scattering for the equation \eqref{main} ?
\end{question}

Here there are three high frequency properties that  play a role, namely the behaviors of $a$, of $q$ and of $\Lambda_0$, $\Lambda_1$, all of which will be assumed to be of symbol type. These need to be considered both at frequencies close to $+\infty$ and at $-\infty$, and the two regions are largely independent.
For convenience only we will not differentiate between the two.
We begin our discussion with several  remarks,
which will play a role both in terms of the model we consider (i.e. the choice of $a$ and $q$)
and the regularity level for the result
(i.e. the choice of $\Lambda_0$ and $\Lambda_1$):

\begin{description}
\item [(i) The behavior of $a$ and $a''$]
The convexity (concavity) of $a$ is associated to dispersion, and plays a critical role. To simplify the notations we will assume $a$ is convex, $a'' > 0$, and  also we will assume some polynomial behavior for $a''$ at infinity,
\begin{equation}\label{choose-a}
a''(\xi) \approx \la\xi\ra^\sigma, \qquad |\xi| \to \infty,
\qquad \sigma \in \R ,
\end{equation}
with symbol type bounds for higher derivatives,
\begin{equation}\label{choose-a-reg}
|\partial^j a''(\xi)| \lesssim \la\xi\ra^{\sigma-j}, 
\qquad j \geq 2.
\end{equation}

Here we distinguish two different scenarios:
\begin{itemize}
    \item The \emph{generalized Klein-Gordon} case, $\sigma < -1$, where $a$ has linear behavior at infinity and the linear problem has finite speed of propagation in the high frequency limit. Here we could further distinguish the range $\sigma \in [-2,-1)$ where $a$ does not have a linear asymptote.   The exact Klein-Gordon problem corresponds to $\sigma = -3$.
    
    \item The \emph{generalized NLS} case $\sigma \geq -1$, where $a$ is superlinear at infinity and we have infinite speed of propagation. The NLS equation in particular corresponds to $\sigma=0$, while mKdV type behavior is associated to $\sigma=1$. 
\end{itemize}
\medskip

\item [(ii) The NLS smallness condition] In the regime of balanced frequency interactions, our problem is well approximated by a cubic NLS problem. There solitons can occur in the focusing case, but not small solitons.
To avoid such a scenario, a smallness condition is required.
A straightforward scaling computation yields the relation 
\begin{equation}\label{no-soliton}
\Lambda_0(\xi) \Lambda_1(\xi) \gtrsim \frac{q(\xi,\xi,\xi)}{a''(\xi)},  
\end{equation}
as necessary in order for scattering to hold.
\medskip 

\item[(iii) The relative size of $\Lambda_0$, $\Lambda_1$]
It is natural to expect the function space $X$ in our result to be stable with respect to dyadic frequency localizations. Commuting $x$ with localizations leads to the requirement
\begin{equation}\label{localization}
\Lambda_1(\xi) \lesssim \la \xi \ra \Lambda_0(\xi).    
\end{equation}
The two norms in \eqref{def:X} will be  close  in scaling in the high frequency limit  when we are close to equality 
in this relation.

\medskip

\item[(iv) The normalization of $q$] Here we observe that our problem admits the invariance 
\[
\Lambda_0 \to b \Lambda_0, 
\qquad
\Lambda_1 \to b \Lambda_1,
\qquad
q(\xi_1,\xi_2,\xi_3) \to b^{-1}(\xi_1-\xi_2+\xi_3) b(\xi_1) b(\xi_2) b(\xi_3)q(\xi_1,\xi_2,\xi_3)
\]
obtained via the substitution $u = B(D) v$.
Because of this, we can normalize $q$ at least in the region of balanced frequency interactions,
\begin{equation}\label{q-symbol}
|q(\xi_1,\xi_2,\xi_3)| \lesssim 1 \quad \text{when} \qquad |\xi_1| \approx |\xi_2| \approx |\xi_3| \approx |\xi_1-\xi_2+\xi_3|.
\end{equation}

\item[(v) Semilinear vs. quasilinear] While the size of $q$
in the balanced region contributes to resonant interactions,
a large size in the imbalanced region may provide a quasilinear
term, for which just looking at the size is not enough to 
even guarantee local well-posedness. In this article 
we will simply avoid this issue, and simply assume that $q$
is bounded everywhere, with symbol type regularity separately 
in each component.
\end{description}

Based on the discussion above, for the results in this section we will 
consider the following set-up for the symbols $a$ and $q$:

\begin{enumerate}[label=(\alph*)]
    \item The symbol $a$ is smooth, convex, with $a''$ as in \eqref{choose-a}, and symbol type regularity.
    \medskip
    
    \item The symbol $q$ is smooth, real on the diagonal, and has 
    the form 
    \begin{equation}
    q(\xi_1,\xi_2,\xi_3) = \qq( \xi_1,\xi_2,\xi_3,\xi_1+\xi_3-\xi_2),   
    \end{equation}
(i.e. the trace of $\qq$ on the diagonal $\xi_1-\xi_2+\xi_3-\xi_4=0$), where $\qq$ is bounded and with separate symbol type regularity in all variables,    
\begin{equation}
  |\partial_{\xi_1}^{\alpha_1} \partial_{\xi_2}^{\alpha_2}
  \partial_{\xi_3}^{\alpha_3} \partial_{\xi_4}^{\alpha_4}
  \qq( \xi_1,\xi_2,\xi_3,\xi_4) | \lesssim \prod_{j=1}^4\la\xi_j\ra^{-\alpha_j}.
\end{equation}
Here $\xi_1-\xi_2+\xi_3$ appears naturally as the output frequency in the trilinear interaction.
\end{enumerate}

Now we turn our attention to the regularity required by our result,
which is determined by the symbols $\Lambda_0$ and $\Lambda_1$.  This will be  chosen to be
 \begin{equation}\label{choose-Lambda}
\Lambda_0 = \la\xi\ra^{s_0}, \qquad \Lambda_1 =  \la\xi \ra^{s_1},
\end{equation}
so that 
\[
\|u\|_{X}^2 = \| u\|_{H^{s_0}}^2 + \|Lu\|_{H^{s_1}}^2.
\]
It remains to discuss the choice of $s_0$ and $s_1$, which we would like to have 
as low as possible. So far, the heuristics above indicate that the following two conditions,
arising from \eqref{no-soliton} and \eqref{localization}, are required:
\begin{equation}\label{choose-s0}
s_0 + s_1 \geq -\sigma, \qquad s_1 \leq s_0+1   . 
\end{equation}
Within this range, we note the best case scenario
\begin{equation}\label{choose-s-ideal}
s_0 = - \frac{\sigma+1}{2}, \qquad s_1 = -\frac{\sigma-1}2.
\end{equation}
Indeed, this would correspond to a scale invariant result in the pure power case 
in the high frequency limit\footnote{ Here, if $\sigma < -1$, then  we can normalize in a Galilean fashion to set $a'(\infty)=0$ before scaling.}.
We retain these values as an ideal but unreachable goal, and seek to at least get close to these values. In particular, it is helpful to allow for  at least a small positive margin in the first inequality in \eqref{choose-s0}, 
 in order to be able to allow for the small power type growth in \eqref{energy}. Even with this proviso, we will only be able to get close to the ideal setting in \eqref{choose-s-ideal} only for $\sigma = -3$  (i.e. exact Klein-Gordon) and for the restricted range  $-1 \leq \sigma \leq 1$ (i.e. weak NLS).

 The conditions in \eqref{choose-s0} above are required by the 
 behavior of balanced interactions. However, managing imbalanced frequency
 interactions imposes further restrictions, which will be reflected in the choices below.
 To summarize, we will distinguish several cases, where $\delta$  stands for a small positive constant:
 
 \bigskip
 \begin{enumerate}[label=(\Roman*)]
     \item \textbf{Weak Klein-Gordon}, $\sigma < -3$. Then we set 
     \[
     s_0= -\sigma-2 \qquad  s_1 = -\sigma -1.
     \]
     \medskip
     \item \textbf{Intermediate Klein-Gordon}, $-3 \leq \sigma < -2$.
  Then we set 
\[
  s_0= 1+\delta, \qquad 
     s_1 =- \sigma-1
 \]    
      \item \textbf{Strong Klein-Gordon} $-2 \leq \sigma < -1$.
  Then we set 
\[
  s_0= -\sigma - 1 + \delta, \qquad 
     s_1 =1. 
\]

     \medskip
     \item \textbf{Weak NLS}, $-1 \leq \sigma \leq 1$.
  Then we set 
  \[
  s_0= -\frac{\sigma+1}2+\delta, \qquad 
     s_1 =-\frac{\sigma-1}2. 
 \]
     \medskip
      \item \textbf{Strong NLS} (or KdV+), $1 < \sigma$.
  Then we set 
 \[
  s_0= -1, \qquad 
     s_1 = 0.   
\]     
     
 \end{enumerate}

 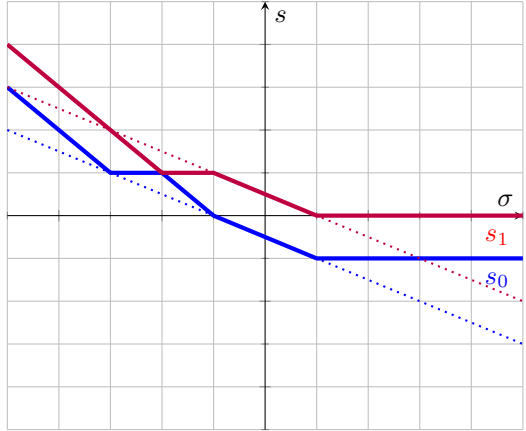
\begin{figure}[h]
\begin{tikzpicture}
\begin{axis}[grid=both,ymin=-5,ymax=5,xmax=5,xmin=-5,xticklabel=\empty,yticklabel=\empty,anchor=origin,
               minor tick num=1,axis lines =middle,xlabel=$\sigma$,ylabel=$s$]
 
\draw[-,ultra thick,blue] (-5,3)--(-3,1)--(-2,1)--(-1,0)--(1,-1)--(5,-1) ;
\draw[-,ultra thick,purple] (-5,4)--(-3,2)--(-2,1)--(-1,1)--(1,0)--(5,0)  ;
\draw[-,thick,dotted,blue] (-5,2)--(5,-3) ;
\draw[-,thick,dotted,purple] (-5,3)--(5,-2)  ;

\node[red] at (axis cs:4.5,-0.5) {$s_1$}; 
\node[blue] at (axis cs:4.5,-1.5) {$s_0$}; 

 \end{axis}


\end{tikzpicture}
\caption{The Sobolev exponents $s_0$ and $s_1$ as a function of $\sigma$.
Dotted lines indicate the best case scenario in \eqref{choose-s-ideal}.}
\end{figure}

Under these assumptions, we have 

\begin{theorem}\label{t:no-comp}
  Assume that the symbols $a,q,\Lambda_0,\Lambda_1$ are as above,
and that the  initial data for our equation \eqref{main}
satisfies:
\begin{equation}\label{small+}
\| u(0)\|_{X} \lesssim \epsilon \ll 1.
\end{equation}
Then the solution exists globally  in time, with 
energy bounds 
\begin{equation}\label{energy+}
\| u(t)\|_{X}  \lesssim \epsilon t^{C \epsilon^2},
\end{equation}
and pointwise decay 
\begin{equation}\label{disp-decay+}
\|\la D \ra^{\frac{\delta}4} u(t)\|_{L^\infty}  \lesssim \frac{\epsilon}{\sqrt t}.
\end{equation} 
\end{theorem}

Here $\delta$ is a small positive parameter, which depends on the choice of $s_0$ and $s_1$
above, and which can be taken to be exactly the one in the choice of $s_0$ in cases 
(II)-(III)-(IV) above. 

\begin{remark}
Our choice of exponents $(s_0,s_1)$ in the cases (I)-(V) above was guided by the goal of coming as close as possible to the 
end-points of the two necessary conditions 
in \eqref{choose-s0}, giving priority 
to the first one. To minimize technicalities
we have fixed the choice of some exponents rather than giving a range. We note however
that increasing $s_0$ while keeping 
$s_1$ fixed is straightforward. This is also
connected with the fact that we have simply assumed that the symbol for the cubic nonlinearity is bounded, rather than imposing various polynomial bounds.
Many of the restrictions arising in the proofs are of technical nature rather 
than fundamental, and arise only in the study of unbalanced interactions, which is 
secondary to our main purpose. We leave it to the reader to explore other variants of these results, as needed.
\end{remark}

One may also supplement Theorem~\ref{t:no-comp} with a matching result on modified scattering and asymptotic completeness,
which exactly mirrors the result provided
in Theorem~\ref{t:scat-model} in the model case. 

In order to best capture the behavior of the asymptotic profile $W$ at high frequencies, it is best to parametrize $W$
by $\xi_v$ rather than by $v$.
To account for this change,  we define 
the asymptotic solution $u_{asympt}$ 
associated to a profile $W$ as
\[
u_{asympt}(t,x) = 
\left\{ 
\begin{array}{lc}
\frac{1}{\sqrt{t}}W(\xi_v) e^{i q(\xi_v,\xi_v,\xi_v) \ln t |W(\xi_v)|^2} e^{it \phi(v)} & \mbox{for } v \in a'(\R) \cr 
0 &\, \mbox{for }  v \not\in a'(\R),
\end{array}
\right. \qquad x = vt ,
\]
where the second alternative occurs only 
in the generalized Klein-Gordon case (I)-(II)-(III).

Then we consider the relation between the initial data $u_0$ and the asymptotic profile $W$, via bounds for the difference
\begin{equation}\label{scat-error+}
e(t,x)=   u(t,x) -  u_{asympt}(t,x) .
\end{equation}

\begin{theorem}
a) For each initial data $u_0$ satisfying the smallness condition \eqref{small+}, there exists an asymptotic profile $W \in H^{1-C\epsilon}_{loc}(\R)$ with the property that 
\begin{equation}\label{which-W}
\| \la \xi \ra^{s_0 + \frac{\sigma}2-C\epsilon^2} W \|_{L^2_{\xi}}
+ 
\| \la \xi \ra^{s_1+ \frac{\sigma}2- C\epsilon^2}  W\|_{H^{1-C\epsilon}_\xi}    \lesssim \epsilon
\end{equation}
for which the above difference satisfies the $L^2$ bounds
\begin{equation}
\| e \|_{L^2_x} \lesssim \epsilon t^{-\delta_1}, \qquad \delta_1 > 0
\end{equation}
as well as the $L^\infty$ bounds
\begin{equation}
\| e \|_{L^\infty} \lesssim \epsilon t^{-\frac12 - \delta_2}, \qquad \delta_2 > 0.
\end{equation}
Furthermore, the map $u_0 \to W$ is injective.

b) For each $W$ satisfying 
\begin{equation} \label{which-W-in}
\| \la \xi \ra^{s_0+\frac{\sigma}2+C\epsilon^2} W \|_{L^2}
+ 
\| \la \xi \ra^{s_1+\frac{\sigma}2-C\epsilon}  W\|_{H^{1+C\epsilon}}    \lesssim \epsilon
\end{equation}
there exists an associated initial data $u_0$ satisfying the smallness condition \eqref{small} so that $W$ is the asymptotic profile of $u_0$.
\end{theorem}
Just as in the case of Theorem~\ref{t:scat-model},  this result is also provided without proof. The proof follows again the same outline as in \cite{NLS}. The exponents in \eqref{which-W},
respectively \eqref{which-W-in} closely bracket
the corresponding exponents in Lemma~\ref{l:gamma}.

To avoid technicalities due to the many cases that would need to be considered, we do not attempt to specify exactly the positive constants $\delta_1$ and $\delta_2$ (which are independent of $\epsilon$).

We remark that the choice of the exponents
$s_0$ and $s_1$, and more precisely the
second bounds  in \eqref{choose-s0} guarantee
that \eqref{which-W} satisfies the pointwise bound
\[
|W(\xi)| \lesssim \epsilon \la \xi\ra^{-\frac{s_0+s_1}2+ \sigma -C \epsilon^2}.
\]
From here, the first bound  in \eqref{choose-s0}, if strict, guarantees that 
\[
\lim_{\xi \to \pm \infty} W(\xi) = 0.
\]
This is particularly interesting in the generalized Klein-Gordon case $\sigma < 1$, where it implies
that the asymptotic solution decays to zero at the edge of its support.

\bigskip

The proof of Theorem~\ref{t:no-comp} follows the same outline as the proof of Theorem~\ref{t:comp}, using a bootstrap argument. The bootstrap assumption will be 
\begin{equation}\label{disp-decay+boot}
\|\la D \ra^{\frac{\delta}8} u(t)\|_{L^\infty}  \lesssim \frac{C\epsilon}{\sqrt t}.
\end{equation}

Using the bootstrap assumption, we first prove the energy bound \eqref{energy+} with $C$ replaced by $C^2$. 
By vector field bounds, the energy estimates will imply a pointwise estimate of the 
form 
\begin{equation}\label{disp-decay+first}
\|\la D \ra^{\frac{\delta}2} u(t)\|_{L^\infty}  \lesssim \frac{\epsilon}{\sqrt t} t^{C^2 \epsilon^2},
\end{equation} 
which would give the bound \eqref{disp-decay+} with an additional $t^{C^2 \epsilon^2}$ loss, but also with a high frequency gain. 
To rectify that, we use our wave packet method to define 
a suitable asymptotic profile $\gamma$, which is then shown to be an approximate solution for the  asymptotic equation. This will allow us to obtain pointwise bounds for the asymptotic profile without the loss, which are the transferred back to $u$.
In the rest of the section, we successively discuss each of the steps of the proof, following the template of the model problem.

\subsection{ Dyadic decompositions} \label{s:dyadic}
Here we motivate and describe the dyadic decompositions that will be used in the sequel. In particular, these will turn out to depend on the ranges for $\gamma$.

\bigskip

\emph{1. The frequency decomposition}.  Here instead of the classical base $2$ dyadic decomposition we will use narrower ranges,
\[
\lambda = (1 + \mu)^m, \qquad m \in \N, \qquad 0 < \mu \ll 1,
\]
with the understanding that at frequencies $\lesssim 1$ we simply split into intervals of size $\mu$. Here $\mu$ is a small universal parameter. The motivation for this choice is to allow for a clean classification of cubic interactions into balanced and unbalanced simply depending on the relative values of $m$. 

We denote the corresponding frequency regions by $I_\lambda^{\pm}$. Here the $\pm$ signs stand for 
positive and negative frequencies, and will be at times omitted if they are not useful. We will also use an adapted partition of unity, again using the $\pm$ superscripts where needed.
\[
1 = \sum_\lambda \nu_\lambda(\xi).
\]

\medskip

\emph{2. The velocity decomposition}. At a given time $t$, we partition the spatial real axis 
corresponding to velocities associated to frequencies in $I_\lambda$. Precisely, 
we denote by $J_\lambda^\pm = a'(I_\lambda^\pm)$ the velocity ranges associated to frequencies in $I_\lambda$, and by $\tilde J_\lambda^\pm$ the corresponding spatial intervals, $\tilde J_\lambda^{\pm} = t J_\lambda^{\pm}$.
We can compute the size of these regions depending on the 
parameter $\gamma$,
\[
|J_\lambda| \approx \lambda a''(\lambda), \qquad  |\tilde J_\lambda| \approx t \lambda a''(\lambda),
\]
where we simply denote $a''(\lambda) \approx |\lambda|^{\sigma}$. Within each interval $J_\lambda$,
respectively $\tilde J_\lambda$ we will choose reference points $v_\lambda$, respectively $x_\lambda$.

Depending on the value of $\sigma$, we distinguish several scenarios:

\begin{enumerate}[label=\alph*)]
    \item The generalized  NLS case, $\sigma \geq -1$. Here $\tilde J_\lambda$ are increasing in size with $\lambda$, and cover the entire real line (except for the degenerate case $\sigma = -1$ where they have equal size). In this case we have an associated spatial partition of unity
\[
1 = \sum_\lambda \chi_\lambda^\pm(x), \qquad \supp \chi_\lambda^\pm \subset 2\tilde J_\lambda^\pm.
\]
\medskip
\begin{figure}[h]
\begin{tikzpicture}

 [shorten >=1pt, node distance=2cm,auto]
             \draw[->] (0,0)--(5,0) node[anchor=north west] {$x$};;
             \draw[-] (0,0)--(3.5,5);
             \draw[-] (0,0)--(2,5);
             \draw[-] (0,0)--(1,5);
             \draw[->] (0,0)--(0,5) node[anchor=south east] {$t$};

             \draw[-] (0,0)--(-5,0);
            \draw[-] (0,0)--(-2,5);
             \draw[-] (0,0)--(-1,5);
             \draw[] (0,0)--(-3.5,5);

             \draw[ultra thick,blue] (0.8,4)--(1.65,4);
               \draw[-] (-4,4)--(4,4);
             \node[blue] at  (1.2,4.3) {$\tilde J_\lambda$};
\end{tikzpicture}
\caption{The velocity decomposition in Case (a), $\sigma \geq -1$: all group velocities are allowed.}
\end{figure}
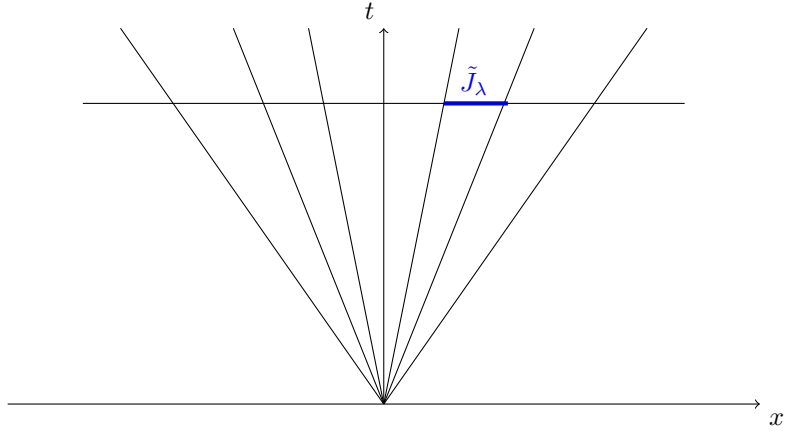
\item The strong Klein-Gordon case,  $-2 \leq \sigma < -1$. 
Here $\tilde J_\lambda$ are decreasing in size with $\lambda$, but their 
sizes $t \lambda a''(\lambda)$ are large enough to dominate the associated uncertainty principle scale $\lambda^{-1}$ as $\lambda \to \infty$. On the other hand, they do not cover the entire  real line, only the range $\tilde J_{in}=(ta'(-\infty),ta'(+\infty))$.
Thus we consider the partition of unity
\[
1 = \sum_\lambda \chi_\lambda^\pm(x) + \chi_{out}(x),
\]
where $\chi_{out}$ is the characteristic function of the outer region $\R \setminus J_{out}$.

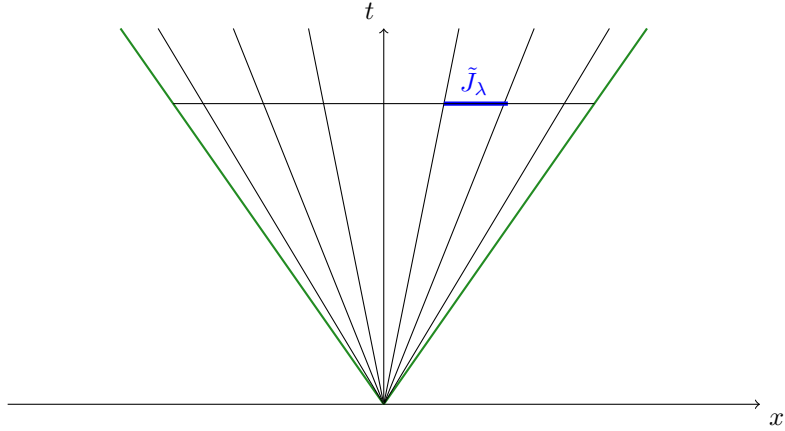
\begin{figure}[h]
\begin{tikzpicture}

 [shorten >=1pt, node distance=2cm,auto]
             \draw[->] (0,0)--(5,0) node[anchor=north west] {$x$};;
             \draw[-,thick,ForestGreen] (0,0)--(3.5,5);
             \draw[-] (0,0)--(2,5);
             \draw[-] (0,0)--(1,5);
             \draw[->] (0,0)--(0,5) node[anchor=south east] {$t$};
            \draw[-] (0,0)--(3,5);
             \draw[-] (0,0)--(-5,0);
            \draw[-] (0,0)--(-2,5);
             \draw[-] (0,0)--(-1,5);
             \draw[-,thick,ForestGreen] (0,0)--(-3.5,5);
            \draw[-] (0,0)--(-3,5);
            
             \draw[ultra thick,blue] (0.8,4)--(1.65,4);
               \draw[-] (-2.8,4)--(2.8,4);
             \node[blue] at  (1.2,4.3) {$\tilde J_\lambda$};
\end{tikzpicture}
\caption{The velocity decomposition in Case (b), $-2 \leq \sigma < -1$: all dispersive waves are localized in an angle. }
\end{figure}

\medskip

\item The Klein-Gordon case $ \sigma <  -2$. 
Here $\tilde J_\lambda$ are also decreasing in size with $\lambda$, but their 
sizes $t \lambda a''(\lambda)$  no longer dominate the associated uncertainty principle scale $\lambda^{-1}$ as $\lambda \to \infty$. For this reason, based on this comparison we define the time dependent threshold $\lambda_0$ by
\begin{equation}\label{lambda0}
t \lambda_0^2 a''(\lambda_0) = 1
\end{equation}
and, depending on $\lambda_0$, we separate into low and high frequencies, and consider the partition of unity
\[
1 = \sum_{\lambda > \lambda_0}  \chi_\lambda^\pm(x) + \chi_{hi}(x) + \chi_{out}(x),
\]
where $\chi_{hi}$ selects a region of size $\lambda_0^{-1}$.
Here the intuition is that up to frequency $\lambda_0$ we see dispersive effects at time $t$,
whereas above that we are simply solving a transport equation at leading order.

\begin{figure}[h]
\begin{tikzpicture}

\begin{axis}[ ymin=0,ymax=7,xmax=7,xmin=-7, xticklabel=\empty,yticklabel=\empty,  axis lines = middle, axis line style={draw=none},  tick style={draw=none}];

 \draw[->,thick] (0,0)--(6,0) node[anchor=south west] {$x$};
             \draw[thick, ForestGreen] (0,0)--(5,5);
             \draw[] (2,2.5)--(4,5);
              \draw[] (.45,1.5)--(1.5,5);
             \draw[] (1.21,1.85)--(3,5);
             \draw[->, thick] (0,0)--(0,5) node[anchor=south east] {$t$};
            \draw[] (3.15,3.5)--(4.5,5);
             \draw[thick] (0,0)--(-5,0);
              \draw[] (-.45,1.5)--(-1.5,5);
             \draw[] (-1.21,1.85)--(-3,5);
            \draw[] (-2,2.5)--(-4,5);
             \draw[thick, ForestGreen] (0,0)--(-5,5);
            \draw[] (-3.15,3.5)--(-4.5,5);
           
             \draw[ultra thick,blue] (1.2,4)--(2.45,4);
               \draw[-] (-4,4)--(4,4);
             \node[blue] at  (1.9,4.4) {$\tilde J_\lambda$};
  \addplot+[no markers]{(x^2+2)^.5} [smooth, blue, mark=none, domain=-3:3] ; 
\end{axis}  
\end{tikzpicture}
\caption{The velocity decomposition in Case (c),  $\sigma < -2$: the dispersive region is above the blue curve. }
\end{figure}
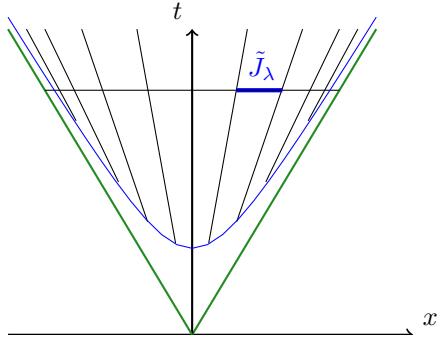
\end{enumerate}

\medskip

\emph{3. The decomposition of $Q$.}
For the trilinear form $Q$, it will be 
very useful to split it into a balanced and an unbalanced component,
\[
Q(u,\bar u,u) = Q^{bal}(u,\bar u, u) + Q^{unbal}(u,\bar u ,u),
\]
depending on the size of the three interacting frequencies. 
Precisely, at the symbol level we set
\[
\qq^{bal}(\xi_1,\xi_2,\xi_3,\xi_4) = \chi_{\la \xi_1 \ra \approx \la \xi_2 \ra
\approx \la \xi_3 \ra   \approx \la \xi_4 \ra   }\qq(\xi_1,\xi_2,\xi_3,\xi_4).
\]
Here the balanced part will play the leading role, and  is the one responsible for the  modified scattering behavior. The unbalanced, part, on the other hand, we will want to treat largely in a perturbative manner. However, some technical difficulties will have to be dealt with along the way.

From the perspective of the spatial Littlewood-Paley decomposition defined earlier,
we will essentially think of the two components as combinations of dyadic 
frequency localizations. Precisely, given dyadic frequencies
$\lambda_j = (1+c)^{m_j}$, we will call the quadruplet $(\lambda_1,\lambda_2,\lambda_3, \lambda_4)$ diagonal if 
$\max\{|m_i-m_j|\} \leq 4 $. We denote the diagonal set of frequencies by $\Diag$.
Then we will simply set
\[
Q^{bal}(u,\bar u, u) = \sum_{\lambda_1,\lambda_2,\lambda_3, \lambda_4 \in \Diag}
P_{\lambda_4} Q(u_1, \bar u_2, u_3) ,
\]
respectively 
\[
Q^{unbal}(u,\bar u, u) = \sum_{\lambda_1,\lambda_2,\lambda_3, \lambda_4 \not\in \Diag}
P_{\lambda_4} Q(u_1, \bar u_2, u_3), 
\]
where for brevity we have denoted $u_i : = P_{\lambda_j} u$. We remark that the $P_{\lambda_4}$
projection can be omitted in the case when $\lambda_4$ is comparable to the highest frequency;
this includes in particular the balanced case.

As a trilinear form applied to $u$, the symbol 
of the expression
\[
P_{\lambda_4} Q(u_1, \bar u_2, u_3) 
\]
has the form
\[
q_{\lambda_1,\lambda_2,\lambda_3,\lambda_4}(\xi_1,\xi_2,\xi_3)= \nu_{\lambda_1}(\xi_1) \nu_{\lambda_2}(\xi_2)
\nu_{\lambda_3}(\xi_3) \nu_{\lambda_4}(\xi_4)
\qq(\xi_1,\xi_2,\xi_3,\xi_4), \qquad \xi_1+\xi_3= \xi_2+\xi_4,
\]
and can be thought of as the diagonal trace of a bump function on the rectangle $I_{\lambda_1} \times I_{\lambda_2} \times I_{\lambda_3} \times I_{\lambda_4}$.
Using separation of variables on this product region,
we can expand these localized symbols as rapidly convergent series
\[
q_{\lambda_1,\lambda_2,\lambda_3,\lambda_4}(\xi_1,\xi_2,\xi_3) = \sum_{k=1}^\infty \nu^k_{1}(\xi_1)\nu^k_{2}(\xi_2)
\nu^k_{3}(\xi_3)\nu^k_{4}(\xi_4)\qquad \xi_1+\xi_3= \xi_2+\xi_4,
\]
where the factors have decaying sizes
\[
| \partial^l \nu^k_{j}| \lesssim k^{-N} \lambda_j^{-l}, 
\qquad l \leq N,
\]
for a large $N$.

Since the dyadic multipliers $\nu^k_j$ are bounded in $X$, 
this will allow us to replace $Q_{\lambda_1,\lambda_2,\lambda_3,\lambda_4}$ in all $X$ bounds with product type operators, precisely of the form
\[
Q_{\lambda_1,\lambda_2,\lambda_3,\lambda_4}
(u,\bar u, u) \approx P_{\lambda_4}(u_1 \bar u_2 u_3).
\]
Furthermore, if $\lambda_1, \lambda_2, \lambda_3 \lesssim  \lambda_4$ then we can further eliminate the outer projection $P_{\lambda_4}$. We will refer to this reduction, later in the paper, as \emph{separation of variables}.

\subsection{The vector field bound}
Our primary goal here is to discuss the counterpart of the vector field estimate 
in Proposition~\ref{p:vf}. We will do this in a frequency localized setting, and also consider the better elliptic bounds outside the corresponding dyadic velocity range. Precisely, we have the following linear estimates:
\begin{proposition}\label{p:vf-gen}
a) 
Let $\delta=s_0+s_1+\sigma > 0$.
Then we have the uniform bound 
\begin{equation}\label{vf-point-lambda}
\|\la D\ra^{\frac{\delta}2-} u\|_{L^\infty} \lesssim \| u\|_{X}.
\end{equation}

b) We also have the dyadic elliptic bounds for a function $u_\lambda$ localized at frequency $\lambda$, and $x_\lambda \in \tilde J_\lambda$:
\begin{equation}\label{vf-ell-ll2}
\| (1-\chi_\lambda) (x-x_\lambda)  u_\lambda \|_{L^2} \lesssim \lambda^{-s_1} \| u_\lambda\|_X ,  
\end{equation}
respectively 
\begin{equation}\label{vf-ell-llinf}
| (1-\chi_\lambda) u_\lambda(x) | \lesssim \frac{\lambda^{-s_1+\frac12}}{ |x-x_\lambda|}     \| u_\lambda\|_X  .
\end{equation}
\end{proposition}

\begin{proof}
Using a dyadic decomposition in frequency as described earlier in Section~\ref{s:dyadic},
\[
u = \sum_\lambda u_\lambda,
\]
we first observe that we can localize the $X$ bound 
and conclude that 
\[
\| u_\lambda \|_{X} \lesssim \| u\|_{X} .
\]
This is where the condition $s_1 \leq s_0+1$ is used.

The advantage is that for each $\lambda$, the size of 
$a''$ is essentially constant, and we may harmlessly extend $a$ to have uniform 
convexity outside $I_\lambda$. 
Hence we will be able to apply directly the results in 
Propositions~\ref{p:vf}, \ref{p:vf-ell}, with the choice of parameters 
\[
R = a''(\lambda) \approx \lambda^\sigma, \qquad M \approx \lambda^{-1}.
\]

a) Since we have $s_0+s_1 \geq -\sigma+\delta$ as well as $s_1 \leq s_0+1$, a direct application of 
Proposition~\ref{p:vf} yields 
\[
\|u_\lambda\|_{L^\infty}^2 \lesssim \frac{1}{t a''(\lambda)}(\lambda^{-{s_0-s_1}} + \lambda^{-2s_0-1}) \|u_\lambda\|_{X}^2 \lesssim   \frac{1}{t \lambda^\delta}\|u_\lambda\|_{X}^2,
\]
which immediately yields the bound \eqref{vf-point-lambda}.

We further remark that, 
in the context of the classification of cases in the previous subsection, in case (c),
which is the Klein-Gordon case, it is also interesting to distinguish the low frequencies
from the high frequencies, and replace the full dyadic decomposition of $u$ by 
\[
u = \sum_{|\lambda| < \lambda_0} u_\lambda + u_{hi},
\]
where the threshold $\lambda_0$ is as in \eqref{lambda0}.
While the above argument applies in all cases, for high frequencies
the desired bound also follows directly from Bernstein's inequality, completely neglecting the $Lu_\lambda$ bound, 
\[
\| u_{hi}\|_{L^\infty} \lesssim \sum_{\lambda > \lambda_0} \lambda^{1+\delta} (a''(\lambda))^\frac12 \|u\|_{H^{s_0}}
\lesssim \lambda_0^{-\delta} t^{-\frac12} \|u\|_{X}.
\]
This is consistent with the fact that in this regime our evolution is at leading order a transport equation, with negligible dispersion. Precisely, in this frequency range we can 
perturbatively replace the symbol $a$  with its affine asymptotes
as $\xi$ approaches $\pm \infty$.

\bigskip
b) Here we similarly apply Proposition~\ref{p:vf-ell}. 
The bound \eqref{vf-ell-ll2}
is obtained directly from \eqref{vf-ell-L2}. 
For \eqref{vf-ell-llinf} we first use a multiplier
$\tilde P_\lambda$ with slightly larger support to 
decompose
\[
(1-\chi_\lambda) (x-x_\lambda)  u_\lambda =
\tilde P_\lambda (1-\chi_\lambda) (x-x_\lambda)  u_\lambda + (1-\tilde P_\lambda)(1-\chi_\lambda) (x-x_\lambda)  u_\lambda
\]
The first term is localized at frequency $\lambda$,
so we can estimate it using Bernstein's inequality 
and \eqref{vf-ell-ll2}, 
\[
\| \tilde P_\lambda (1-\chi_\lambda) (x-x_\lambda)  u_\lambda \|_{L^\infty} \lesssim 
\lambda^{\frac12} \| (1-\chi_\lambda) (x-x_\lambda)  u_\lambda \|_{L^2}
\lesssim  \lambda^{-s_1+\frac12}    \| u_\lambda\|_X.
\]
In the second term the coefficient $(1-\chi_\lambda) (x-x_\lambda)$ must be localized at frequency at least $\lambda$,
\[
 (1-\tilde P_\lambda)(1-\chi_\lambda) (x-x_\lambda)  u_\lambda = 
  (1-\tilde P_\lambda) g_{\gtrsim \lambda}   u_\lambda, \qquad g = (1-\chi_\lambda) (x-x_\lambda).
\]
Then we estimate 
\[
\begin{aligned}
\| (1-\tilde P_\lambda)(1-\chi_\lambda) (x-x_\lambda)  u_\lambda \|_{L^\infty}
\lesssim & \  \| g_{\gtrsim \lambda}   u_\lambda\|_{L^\infty}
\\
\lesssim  & \ \lambda^\frac12 \| g_{\gtrsim \lambda}   u_\lambda\|_{L^2} + \lambda^{-\frac12}
\| \partial_x [g_{\gtrsim \lambda}   u_\lambda]\|_{L^2}
\\
\lesssim & \ (\lambda^\frac12 \| g_{\gtrsim \lambda}\|_{L^\infty}  + \lambda^{-\frac12}
\| \partial_x g_{\gtrsim \lambda}\|_{L^\infty} ) \|u_\lambda\|_{L^2}
\\
\lesssim & \ \lambda^{-s_0-\frac12} \| \partial_x g\|_{L^\infty} \|u_\lambda\|_{X}
\end{aligned}
\]
which suffices since $\| \partial_x g\|_{L^\infty}
\lesssim 1$ and $s_0 \geq s_1-1$.

\subsection{ Bounds for $Q$ and the energy estimate for $u$}
Here the first goal is to prove the following energy bound for the function $u$:

\begin{proposition}\label{p:ee-u}
Assume that $u$ is a solution  to \eqref{main}, 
under the same assumptions as in Theorem~\ref{t:no-comp}. Then we have the bound
\begin{equation}
\frac{d}{dt} \| u(t) \|_{H^{s_0}}^2 \lesssim \|\la D \ra^{\frac{\delta}8} u\|_{L^\infty}^2 
\|u\|_{X}^2.
\end{equation}
\end{proposition}
We note that in many problems this bound is independent on the $X$ norm, 
and has instead  the form
\begin{equation}
\frac{d}{dt} \| u(t) \|_{H^{s_0}}^2 \lesssim \|\la D \ra^{\frac{\delta}8} u\|_{L^\infty}^2 
\|u\|_{H^{s_0}}^2.
\end{equation}
This is the case if $s_0 \geq 0$ (see the proof below) but also if $Q$ has additional structure.

\begin{proof}
Differentiating in time and using the equation \eqref{main}, this reduces to the weighted inequality 
\begin{equation} \label{q1}
\|Q(u,\bar u,u)\|_{H^{s_0}} 
\lesssim \|\la D \ra^{\frac{\delta}8} u\|_{L^\infty}^2 
\|u\|_{X}.
\end{equation}

Here we distinguish two cases depending on the sign of $s_0$:
\medskip

i) $s_0 \geq 0$. Here we have the simpler bound
\begin{equation} \label{q1-easy}
\|Q(u,\bar u,u)\|_{H^{s_0}} 
\lesssim \|\la D \ra^{\frac{\delta}8} u\|_{L^\infty}^2 \| u\|_{H^{s_0}},
\end{equation}
which does not involve any control for $Lu$.
Since  $Q$ satisfies the 
symbol bounds \eqref{q-symbol}, this easily follows 
by a standard Littlewood-Paley decomposition with respect to all 
inputs and the output. The $H^{s_0}$ factor on the right is 
always chosen to correspond to the highest frequency, and the $\delta$ exponent readily ensures dyadic summation. More precisely, writing
\[
 Q(u,\bar u,u) = 
 \sum_{\lambda_1,\lambda_2,\lambda_3,\lambda_4} P_{\lambda_4} Q(u_{\lambda_1},\bar u_{\lambda_2}, u_{\lambda_3}),
 \]
and relabeling increasing order $\{ \lambda_1,\lambda_2, \lambda_3\} = \{\lambda_{lo}, \lambda_{mid}, \lambda_{hi}\}$
we must have either 
\begin{enumerate}[label=(\alph*)]
\item $\lambda_4 \approx \lambda_{hi}$, or

\item  $\lambda_4 < \lambda_{mid} \approx \lambda_{hi}$,
\end{enumerate}
and correspondingly decompose $Q=Q_a + Q_b$.

For $Q_a$ we may use orthogonality to estimate
\[
\| Q_a(u,\bar u,u)\|_{H^{s_0}}^2 \lesssim  \sum_{\lambda_{hi}}
\left(\sum_{\lambda_{lo},\lambda_{mid}} \| u_{\lambda_{lo}}\|_{L^\infty} \| u_{\lambda_{mid}} \|_{L^\infty}\right)^2 \| u _{\lambda_{hi}}\|^2_{H^{s_0}},
\]
where the inner sum is estimated by the $L^\infty$ norm in 
\eqref{q1-easy}.

For $Q_b$ on the other hand we neglect orthogonality
and estimate directly 
\[
\| Q_b(u,\bar u,u)\|_{H^{s_0}}^2 \lesssim  \sum_{\lambda_{lo}} \sum_{\lambda_{mid} \approx \lambda_{hi}}
\| u_{\lambda_{lo}}\|_{L^\infty} \| u_{\lambda_{mid}} \|_{L^\infty} \| u _{\lambda_{hi}}\|_{H^{s_0}},
\]
where the summation with respect to the two indices 
is again guaranteed  by the $L^\infty$ norm in 
\eqref{q1-easy}.

\medskip

ii) $s_0 < 0$, which is needed only in the generalized NLS case $\sigma \geq -1$. In this case, the bound \eqref{q1-easy}
applies only to the portion of $Q$ where at least one of the three 
input frequencies, which we denote by $\lambda_1$, $\lambda_2$ and $\lambda_3$, is at most comparable to the output frequency $\lambda_4$.

Hence, from here on we assume that $\lambda_4 \ll  \lambda_j$, $j = 1,2,3$. 
This guarantees that $\lambda_1$, $\lambda_2$ and $\lambda_3$ should all be distinct, and also the largest two should be comparable. Under these 
assumptions, it remains to prove the estimate
\begin{equation}
\begin{aligned}
\lambda_4^{s_0} \|P_{\lambda_4} Q(u_{\lambda_1},\bar u_{\lambda_2},u_{\lambda_3})\|_{L^2} 
\lesssim & \ 
\| u_{ \lambda_1} \|_{L^\infty} \| u_{ \lambda_2} \|_{L^\infty} \| u_{\lambda_3}\|_{X}+
 \| u_{ \lambda_2} \|_{L^\infty} \| u_{ \lambda_3} \|_{L^\infty} \| u_{\lambda_1}\|_{X}
 \\ 
&  + \| u_{ \lambda_3} \|_{L^\infty} \| u_{ \lambda_1} \|_{L^\infty}
\| u_{\lambda_2}\|_{X}.
\end{aligned}
\end{equation}
Here we note that, since the two highest frequencies are comparable,
the dyadic summation with respect to the four frequencies is straightforward using the $\delta$ factor, and \eqref{q1} follows.

To prove the last bound,  we retain the restrictions on $\lambda_1$, $\lambda_2$
and $\lambda_3$, but then harmlessly drop the projection $P_{\lambda_4}$. Then we can use separation of variables and reduce the problem to the product case, where 
it suffices to show that
\begin{equation}
\begin{aligned}
 \|u_{\lambda_1}\bar u_{\lambda_2} u_{\lambda_3}\|_{L^2} 
\lesssim & \ 
\| u_{ \lambda_1} \|_{L^\infty} \| u_{ \lambda_2} \|_{L^\infty} \| u_{\lambda_3}\|_{X}+
 \| u_{ \lambda_2} \|_{L^\infty} \| u_{ \lambda_3} \|_{L^\infty} \| u_{\lambda_1}\|_{X}
 \\ 
&  + \| u_{ \lambda_3} \|_{L^\infty} \| u_{ \lambda_1} \|_{L^\infty}
\| u_{\lambda_2}\|_{X}.
\end{aligned}
\end{equation}

Next, we separate the product with respect to dyadic velocity ranges. Since the $\lambda$'s cannot be all equal, it suffices 
to estimate the expression
\[
I= \|(1-\chi_{\lambda_1}) u_{\lambda_1}\bar u_{\lambda_2}u_{\lambda_3}\|_{L^2} .
\]
By Proposition~\ref{p:vf-gen} we have
\[
\begin{aligned}
I \lesssim &\ \|(1-\chi_{\lambda_1)}u_{\lambda_1}\|_{L^2}\| u_{\lambda_2}\|_{L^\infty} \|u_{\lambda_3}\|_{L^\infty}
\\
\lesssim & \ \sup_{\lambda \neq \lambda_1} \frac{\lambda_1^{-s_1}}{t |a'(\lambda_1) - a'(\lambda)|}
\|u_{\lambda_1}\|_{X} \|u_{\lambda_2}\|_{L^\infty} \|u_{\lambda_3}\|_{L^\infty}
\\
\lesssim & \ \frac{\lambda_1^{-s_1}}{t  \lambda_1 |a''(\lambda_1)|}
\|u_{\lambda_1}\|_{X} \|u_{\lambda_2}\|_{L^\infty} \|u_{\lambda_3}\|_{L^\infty},
\end{aligned}
\]
where the $\lambda$ dependent weight is maximized when $\lambda$ is near $\lambda_1$. Then it suffices to check that
\[
\frac{\lambda_1^{-s_1}}{t \lambda_1 |a''(\lambda_1)|} \lesssim 1.
\]
Given the choice of $s_1$ and that $\sigma \geq -1$, this is true with a substantial gain. Thus the proof of the Proposition is complete.
\end{proof}

A second objective here is to show that, in the context of the balanced/unbalanced
decomposition for the cubic nolinearity $Q$, we have a better bound for the unbalanced part. 
This bound will play a role in our wave packet testing in the next subsection,
precisely in the estimate for the error  in the asymptotic equation.

\begin{proposition}
The unbalanced part $Q^{unbal}$ of $Q$ satisfies the better $L^\infty$ bound
\begin{equation}\label{Q-L2-from-Lh}
\|\chi_\lambda P_\lambda Q^{unbal}(u,\bar u ,u) \|_{L^\infty} \lesssim 
\frac{\lambda^{-\frac{\delta}{4}}}{t^{\frac32+\frac{\delta}4}}  \|u\|_{X}^3,
\end{equation}  
provided that either $ \sigma \geq -2$
or  \{$\sigma < -2$ and $t \lambda^{\sigma+2}
\geq 1$\}.
\end{proposition}

We remark that, depending on $\sigma$ and on the balance of the three frequencies,
in some of the cases one can get a better asymptotic equation error bound by 
using $L^2$ estimates for $Q^{unbal}$. We do not pursue this here because it is not needed.

\begin{remark}
This bound is needed 
in order to be able to control the contribution of $Q^{unbal}$ to the error in the wave packet testing. Precisely, we will need to be able to verify that 
\[
\la Q^{unbal}(u,\bar u, u), \uu_v \ra \lesssim t^{-1-\delta}
\]
for $v \in J_\lambda$, and $\lambda < \lambda_0$ in the case $\sigma < -2$.
This requires the $L^\infty$ bound
\[
\| \chi_\lambda  P_\lambda Q \|_{L^\infty} \lesssim t^{-\frac32-}.
\]

\end{remark}

\begin{proof}
We first simply consider a triple product $u_1 \bar u_2 u_3$ where $\lambda_1,\lambda_2$ and $\lambda_3$ are not all equal, and estimate it within a dyadic velocity region $A_\lambda$. For that
we apply  \eqref{vf-ell-ll2} and \eqref{vf-point-lambda} for a $\lambda_j$, say $\lambda_3$, which is away from $\lambda$.
This yields
\begin{equation}\label{Q-2} 
\|\chi_\lambda u_1 \bar u_2 u_3 \|_{L^2} \lesssim 
\frac{\lambda_1^{-\frac{\delta}{2}} \lambda_2^{-\frac{\delta}{2}}\lambda_3^{-s_1}}{t^2 |a'(\lambda_3)- a'(\lambda)|}  \|u\|_{X}^3 ,
\end{equation}
respectively
\begin{equation}\label{Q-inf}
\|\chi_\lambda u_1 \bar u_2 u_3 \|_{L^\infty} \lesssim 
\frac{\lambda_1^{-\frac{\delta}{2}} \lambda_2^{-\frac{\delta}{2}} \lambda_3^{-s_1+\frac12}}{t^2 |a'(\lambda_3)- a'(\lambda)|}  \|u\|_{X}^3 .
\end{equation}
We complement these with the trivial bound
\begin{equation}\label{Q-inf+}
\|\chi_\lambda u_1 \bar u_2 u_3 \|_{L^\infty} \lesssim 
\frac{\lambda_1^{-\frac{\delta}{2}} \lambda_2^{-\frac{\delta}{2}} \lambda_3^{-s_0+\frac12}}{t}  \|u\|_{X}^3 .
\end{equation}
To use these estimates we consider two scenarios:
\smallskip

(i) $\lambda_1=\lambda$ and $\lambda_3 < \lambda_2 \ll \lambda$. Then 
we can separate variables to discard $P_\lambda$, and apply the above 
bounds \eqref{Q-inf} and \eqref{Q-inf+}. 
  Now we examine the coefficient 
 in \eqref{Q-inf}   as a function of $\lambda_3$.  For $\sigma \geq -1$ we get 
\[
\|\chi_\lambda u_1 \bar u_2 u_3 \|_{L^\infty} \lesssim \lambda_1^{-\frac{\delta}{2}} \lambda_2^{-\frac{\delta}{2}} \lambda_3^{-s_1+\frac12} \lambda^{-\sigma -1} t^{-2}  \|u\|_{X}^3,
\]
which suffices. For $\sigma < -1$ we get 
\[
\|\chi_\lambda u_1 \bar u_2 u_3 \|_{L^\infty} \lesssim \lambda_1^{-\frac{\delta}{2}} \lambda_2^{-\frac{\delta}{2}} \lambda_3^{-s_1-\frac12-\sigma } t^{-2}  \|u\|_{X}^3.
\]
This still suffices directly in the range 
$-\frac32 \leq \sigma \leq -1$, 
and after interpolation with \eqref{Q-inf+}
in the remaining range $ \sigma < -\frac32$.
In all cases  the summation in $\lambda_3$ and $\lambda_2$ is straightforward.

\smallskip
(ii) In the remaining case we must have at least two comparable high frequencies, say
$\lambda_2, \lambda_3 \gtrsim \lambda$, one of which, say $\lambda_3$, is separated from $\lambda$.  Then we replace the cutoff $\chi_\lambda$ by 
one with a double support, call it  $\tchi_\lambda$, which equals one on a 
comparably sized neighbourhood of the support of $\chi_\lambda$. Precisely, we write
\[
\chi_\lambda P_\lambda = \chi_\lambda P_\lambda \tchi_\lambda +  \chi_\lambda P_\lambda (1-\tchi_\lambda).
\]
The second term is easily taken care of by noting that 
\[
\| \chi_\lambda P_\lambda (1-\tchi_\lambda) \|_{L^\infty \to L^\infty}
\lesssim \frac{1}{(t a''(\lambda) \lambda^2)^N}
\]
combined with the pointwise bound for each of the factors.

For the first term we apply 
\eqref{Q-2}, noting that the coefficient is nonincreasing in $\lambda_3 \gtrsim \lambda$. For $\sigma \geq -1$ we obtain
\begin{equation*}\label{Q-2-bis}
\|\tchi_\lambda u_1 \bar u_2 u_3 \|_{L^2} \lesssim 
\frac{\lambda_3^{-s_1}}{t^2 \lambda_3^{1+\sigma}}  \|u\|_{X}^3 ,
\end{equation*}
and conclude using Bernstein's inequality at frequency  $\lambda$.
For $\sigma < -1$ we obtain
\begin{equation*}
\|\tchi_\lambda u_1 \bar u_2 u_3 \|_{L^2} \lesssim 
\frac{\lambda_3^{-s_1}}{t^2 \lambda^{1+\sigma}}  \|u\|_{X}^3 ,
\end{equation*}
Then we use Bernstein's inequality at frequency  $\lambda$ and interpolate with 
\eqref{Q-inf+} as in case (i).

\end{proof}

\subsection{ The energy estimate for $Lu$}
Here the objective is to prove the energy estimate for $Lu$.
As in the model case, this will be achieved via 
a cubic correction $C$ so that we can obtain a favorable estimate
for the nonlinear expression
\[
L^{NL} u = Lu+ t C(u,\bar u,u).
\]
Precisely, we will prove the following

\begin{proposition}\label{p:ee-Lu}
There exists a trilinear, translation invariant  correction $C$ with the following properties

(i) Uniform bound for $C$:
\begin{equation}\label{L2-C}
\|C(u,\bar u,u) \|_{H^{s_1}} \lesssim   \|\la D \ra^{\frac{\delta}8} u\|_{L^\infty}^2  \| u\|_{H^{s_0}} .
\end{equation}

(ii) Energy bound for $L^{NL} u$,
\begin{equation}\label{L2-LNL}
\frac{d}{dt} \|L^{NL} u\|_{H^{s_1}}^2
\lesssim \| u\|_{X}^2  \|\la D \ra^{\frac{\delta}8} u\|_{L^\infty}^2 
+ t^{-\frac12-\delta} \| u\|_{X}^3  \|\la D \ra^{\frac{\delta}8} u\|_{L^\infty} .
\end{equation}
\end{proposition}
One immediate consequence of \eqref{L2-C} combined with the 
bootstrap assumption \eqref{disp-decay+boot} is the 
norm equivalence
\begin{equation}\label{same-X}
\| u\|_X^2 \approx \| u\|_{H^{s_0}}^2 + \| L^{NL} u\|_{H^{s_1}}^2.
\end{equation}
Using this property one easily sees that, combining the energy estimates for $u$ and $L^{NL} u$ in Propositions~\ref{p:ee-u}, \ref{p:ee-Lu}, and using the bootstrap assumption \eqref{disp-decay+boot}, we obtain 
by Gronwall's inequality the energy estimate in \eqref{energy+}.

\begin{proof}
For the expression $w:=L^{NL} u$ we have an equation of the form 
\[
(i\partial_t - A) w = LQ(u,\bar u,u) +t R_3(u,\bar u,u)  + i C(u,\bar u,u)
+ t R_5(u,\bar u,u,\bar u,u),
\]
where $R_3$ has symbol
\[
r_3(\xi_1,\xi_2,\xi_3) = c(\xi_1,\xi_2,\xi_3)( a(\xi_1)-a(\xi_2) + a(\xi_3)
- a(\xi_1-\xi_2+\xi_3)),
\]
and $R_5$ is simply the quintilinear form arising from the time derivative of $C$.

The objective is then to choose the correction $C$ so that \eqref{L2-C} holds, and 
we can estimate the source terms in $H^{s_1}$,
\begin{equation}\label{c2}
\|LQ(u,\bar u ,u)+ tR_3(u,\bar u,u) \|_{H^{s_1}} \lesssim \|u\|_{X} \|\la D \ra^{\frac{\delta}8} u\|_{L^\infty}^2  ,
\end{equation} 
respectively
\begin{equation}\label{c3}
\| R_5(u,\bar u,u,\bar u,u) \|_{H^{s_1}} \lesssim \|u\|_{H^{s_0}} \|\la D \ra^{\frac{\delta}8} u\|_{L^\infty}^4 .
\end{equation}

Here naively one may hope to use the same correction $C$ 
as in the compact case, so that we have
\[
LQ(u,\bar u,u) +t R_3(u,\bar u,u)= Q(Lu,\bar u,u) - 
Q(u,\overline{Lu},u) + Q(u,\bar u,L u).
\]
However, as it turns out, there are some difficulties which such a direct 
approach. Precisely,  considering a full dyadic decomposition for $Q$,
there are two interesting scenarios to consider:

\medskip

a) Balanced interactions, where the three input frequencies and the output frequency are all comparable, say to a fixed frequency $\lambda$. Then the symbol $c$ has 
similar support, symbol type regularity and size 
\[
|c(\xi_1,\xi_2,\xi_3)| \lesssim \lambda^{-1} |q(\xi_1,\xi_2,\xi_3)| \lesssim |\lambda|^{-1}.
\]
In this case the bounds \eqref{L2-C} and \eqref{c2} are straightforward, nothing but a rescaled version of the corresponding bounds in the compact case.
We still  need to prove \eqref{c3}, which contains some unbalanced interactions, but this is not so difficult.

\medskip

b) Unbalanced interactions, where, instead, the use of the correction $C$ 
would cause trouble:
\begin{itemize}
    \item The expression of $C$ would be more complicated, which
causes difficulties with \eqref{L2-C} and \eqref{c3}.
\item the bound \eqref{c2} is unbalanced, which causes difficulties unless 
$s_1=0$ or we have a favourable frequency balance.
\end{itemize}

However, the redeeming feature in this case is that, in each dyadic velocity range, at least one of the three inputs must correspond to a different range
of velocities, so the corresponding frequency localized operator $L$ is elliptic there. It follows that  the expression $LQ(u,\bar u,u)$ no longer needs to be corrected, and instead should be estimated directly, in an elliptic fashion.

To implement the heuristic strategy described above, we 
decompose $Q$ into a balanced and an unbalanced component,
\[
Q(u,\bar u,u) = Q^{bal}(u,\bar u, u) + Q^{unbal}(u,\bar u ,u),
\]
where at the symbol level we set
\[
\qq^{bal}(\xi_1,\xi_2,\xi_3,\xi_4) = \chi_{\la \xi_1 \ra \approx \la \xi_2 \ra
\approx \la \xi_3 \ra   \approx \la \xi_4 \ra   }\qq(\xi_1,\xi_2,\xi_3,\xi_4).
\]
Then we choose the normal form correction $C=C^{bal}$ to account for the balanced term, where the corresponding errors are estimated as discussed above. On the other hand, the unbalanced term we simply treat perturbatively, without any correction. 

\bigskip

\textbf{ A. The balanced term.}
To account for the balanced term, we follow the compact case and set
\begin{equation}
    c^{bal}(\xi_1,\xi_2,\xi_3) = q^{bal}(\xi_1,\xi_2,\xi_3) \frac{a_\xi(\xi_1) - a_\xi(\xi_2) + a_\xi(\xi_3)- a_\xi(\xi)}{ a(\xi_1)-a(\xi_2) + a(\xi_3) - a(\xi)},
\end{equation}
so that we have the algebraic relation
\[
L Q^{bal}(u,\bar u,u) - t  R_3^{bal}(u,\bar u,u) =
Q^{bal}(Lu,\bar u,u) - C^{bal}(u, \overline{Lu},u) + C^{bal}(u,\bar u, Lu)
+  D(u,\bar u, u),
\]
with 
\[
d(\xi_1,\xi_2,\xi_3) = i ( \partial_{\xi_1} -\partial_{\xi_2} + \partial_{\xi_3})
q^{bal}(\xi_1,\xi_2,\xi_3).
\]

Then we have 
\begin{lemma}
The above correction $C^{bal}$ satisfies the estimates \eqref{L2-C}, \eqref{c2} and \eqref{c3}.
\end{lemma}

\begin{proof}
As mentioned earlier, the proof of \eqref{L2-C}, \eqref{c2} is simply a rescaled
version of the similar argument in Section~\ref{s:energy}. As such, it is omitted
and left as an exercise for the reader.

The bound \eqref{c3}, on the other hand, involves also some unbalanced interactions and deserves some separate attention. Localizing in frequency and separating variables, we split 
\[
C^{bal} = \sum_\lambda C^{bal}_\lambda, 
\]
where we 
can assume that the frequency $\lambda$ portion $C^{bal}_\lambda$ of $C^{bal}$ has the form
\[
C^{bal}_\lambda (u,\bar u,\bar u) = \lambda^{-1} u_\lambda \bar u_\lambda u_\lambda .
\]
Then the corresponding component of $R_5$ has terms of the form
\[
R_{5,\lambda} (u,\bar u,u,\bar u,u)= \lambda^{-1} u_\lambda \bar u_\lambda P_\lambda Q(u,\bar u,u) .
\]
Hence, we can bound it by
\[
\|R_{5,\lambda}(u,\bar u,u,\bar u,u) \|_{H^{s_1}} \lesssim 
\lambda^{s_1-1} \|u_\lambda\|_{L^2} \| u_\lambda \|_{L^\infty} \| Q(u,u,u)\|_{L^\infty} \lesssim  \|u_\lambda\|_{H^{s_0}} 
\|\la D \ra^{\frac{\delta}8} u\|_{L^\infty}^4. 
\]

\end{proof}

\bigskip

\textbf{B. The unbalanced term.} This corresponds to the unbalanced  component $Q^{unbal}$ of $Q$.
Here we set our correction to $0$, so that $R_3$ and $R_5$ also vanish.
Then it remains to prove that we have the following result:
\begin{lemma}\label{l:LQunbal}
Assume that $s_0, s_1$ are chosen as in Theorem~\ref{t:no-comp}. Then 
for $\delta > 0$ we have the following $L^2$ type bound:
\begin{equation}\label{c2+}
\|L Q^{unbal}(u,\bar u ,u) \|_{H^{s_1}} \lesssim \|u\|_{X} 
\|\la D \ra^{\frac{\delta}8} u\|_{L^\infty}^2 + t^{-\frac12 -\delta} \|u\|_{X}^2 
\|\la D \ra^{\frac{\delta}8} u\|_{L^\infty}.
\end{equation}   
\end{lemma}
Here the advantage is that we can choose which inner $u$ we place the $L$
on. Using the bound \eqref{c2+} in the lemma, the proof of the $L^2$ energy bound \eqref{L2-LNL} for $L^{NL} u$ is concluded.

\begin{proof}
 We localize the cubic expression $Q$ in frequency to dyadic regions 
 associated with input frequencies $\lambda_1$, $\lambda_2$, $\lambda_3$ and output frequency $\lambda_4$. Since the choice of the small parameter $\delta$ is flexible,
 the dyadic summation with respect to $\lambda_1$, $\lambda_2$, $\lambda_3$ and $\lambda_4$ is straightforward. For this reason, it suffices to prove the lemma
 in the case when  $\lambda_1$, $\lambda_2$, $\lambda_3$ and $\lambda_4$ are fixed.
To streamline notations, we will denote $u_j:= P_{\lambda_j} u$ for $j = 1,2,3$.

In each such region, the nonlinear expression $Q$ is essentially like a product, which then gets localized to the output frequency $\lambda_4$. Using separation of variables,
we can reduce the problem to the case 
\[
Q^{unbal}(u_1,\bar u_2,u_3) = P_{\lambda_4}(u_1 \bar u_2 u_3),
\]
where $\lambda_1,\lambda_2,\lambda_3$ are not all equal. 
This key property implies that, in a given a dyadic velocity range associated to a frequency $\lambda$, we must have at least one of the three $L$'s act as an elliptic operator; our estimate is primarily based on this principle.
We will further separate the problem into two cases, depending on the relative size of the three interacting frequencies $\lambda_1$, $\lambda_2$ and $\lambda_3$.
\bigskip

a)  The $llh$ case, where 
\[
\lambda_1 \leq \lambda_2 \ll  \lambda_3 ,
\]
or the symmetric case where $\lambda_1$ and $\lambda_3$
are interchanged. In this case we must have $\lambda_4 \approx \lambda_3$, and we can also use separation of variables to discard the $P_{\lambda_4}$ projector. Further, it will be convenient 
to commute $L$ inside, and write
\[
L (u_1 \bar u_2 u_3) = u_1 \bar u_2 L u_3 +   t R(u_1,\bar u_2,u_3),
\]
where the symbol of $R$ is 
\[
r(\xi_1,\xi_2,\xi_3) = a'(\xi_1-\xi_2+\xi_3) - a'(\xi_3).
\]
This is a smooth symbol in all three variables on the corresponding dyadic scales, and has size
\[
|r(\xi_1,\xi_2,\xi_3)| \lesssim \lambda_2 a''(\lambda_3).
\]
The first term is estimated in a straightforward fashion,
\[
\| u_1 \bar u_2 L u_3\|_{H^{s_1}} \lesssim \|u_1\|_{L^\infty}\|u_2\|_{L^\infty} \|Lu_3\|_{H^{s_1}}.
\]
For the second term we can use separation of variables to drop again the multipliers, and we are left with the task
of estimating the expression
\[
t \lambda_3^{s_1}\lambda_2 a''(\lambda_3)
\| u_1 \bar u_2 u_3\|_{L^2}.
\]
Finally, we use a spatial angular localization to separate into directions associated to a dyadic 
frequency $\lambda$. In this case we can consider a very simple separation, depending on whether the direction $\lambda$ matches $\lambda_3$ or not, writing
\[
u_1 \bar u_2 u_3 = \chi_{\lambda_3} ( u_1 \bar u_2 u_3)+ (1-\chi_{\lambda_3}) ( u_1 \bar u_2 u_3).
\]
For the first term we use \eqref{vf-ell-ll2} for $u_2$ 
in order to write
\[
\begin{aligned}
t \lambda_3^{s_1}\lambda_2 a''(\lambda_3) 
\| \chi_{\lambda_3} ( u_1 \bar u_2 u_3)\|_{L^2}
\lesssim & \ 
t \lambda_3^{s_1}\lambda_2 a''(\lambda_3) 
\| \chi_{\lambda_3}  u_2\|_{L^2} \| u_1\|_{L^\infty} \| u_3\|_{L^\infty}
\\
\lesssim & \ 
t \lambda_3^{s_1}\lambda_2 a''(\lambda_3) \frac{1}{ \lambda_2^{s_1} t |a'(\lambda_2) - a'(\lambda_3)|}
\| u_2\|_{X} \| u_1\|_{L^\infty} \| u_3\|_{L^\infty}.
\end{aligned}
\]
Here the coefficient on the right is nondecreasing
in $\lambda_2$ in all cases (this corresponds to the restriction $s_1 \leq 1$ if $\sigma \geq -1$, respectively $s_1 \leq -\sigma$ if $\sigma < -1$, which are satisfied for our choice
of exponents) and equals $1$ if $ \lambda_2 = \lambda_3$.

For the second term we instead use \eqref{vf-ell-ll2} for $u_2$ 
in order to write
\[
\begin{aligned}
t \lambda_3^{s_1}\lambda_2 a''(\lambda_3) 
\| (1-\chi_{\lambda_3}) ( u_1 \bar u_2 u_3)\|_{L^2}
\lesssim & \ 
t \lambda_3^{s_1}\lambda_2 a''(\lambda_3) 
 \| u_1\|_{L^\infty} \| u_3\|_{L^\infty}
 \| (1-\chi_{\lambda_3})  u_2\|_{L^2}
\\
\lesssim & \ 
t \lambda_3^{s_1}\lambda_2 a''(\lambda_3) \frac{1}{ \lambda_3^{s_1} t \lambda_3 |a''(\lambda_3)|}
\| u_2\|_{X} \| u_1\|_{L^\infty} \| u_3\|_{L^\infty}
\\
\lesssim & \ 
\| u_2\|_{X} \| u_1\|_{L^\infty} \| u_3\|_{L^\infty}.
\end{aligned}
\]
This concludes the proof of \eqref{c2+} in this case.

\bigskip

 \bigskip

b) The $lhh$ case,
\[
\lambda_1 < \lambda_2 \approx  \lambda_3, \qquad \lambda_4 \lesssim \lambda_2.
\]
 or permutations thereof.  Here we have many subcases to consider. We first reduce their number by peeling off some of the easier ones.

A first argument we can apply is to simply write 
\begin{equation}\label{LP3}
L P_{\lambda_4} (u_1 \bar u_2 u_3) = P_{\lambda_4} 
(L u_1 \bar u_2  u_3) + [x,
P_{\lambda_4}] (u_1 \bar u_2 u_3) + t R(u_1,\bar u_2,u_3) ,
\end{equation}
where the commutator term is essentially of the form $ \lambda_4^{-1}
P_{\lambda_4} (u_1 u_2 u_3)$ and the remainder $R$ arises from switching
the argument of $L$, and has symbol 
\[
r(\xi_1,\xi_2,\xi_3) = a'(\xi_1)-a'(\xi_1-\xi_2+\xi_3).
\]
This is a smooth symbol on the four associated dyadic scales, and of size
\[
|r(\xi_1,\xi_2,\xi_3)| \lesssim  |a'(\lambda_1)-a'(\lambda_4)|.
\]

Here we can estimate the first term in \eqref{LP3} in $H^{s_1}$
by 
\[
\| P_{\lambda_4}(L u_1 \bar u_2  u_3)\|_{H^{s_1}}
\lesssim \|L u_1 \|_{H^{s_1}} \|u_2\|_{L^\infty} \|u_3\|_{L^\infty},
\]
provided that  
\begin{equation}
\text{either} \quad \lambda_4 \lesssim \lambda_1 \quad \text{or} \quad \sigma \geq 1,    
\end{equation} 
where the second condition ensures that $s_1=0$.

The second term in \eqref{LP3} is estimated by 
\[
\lambda_4^{-1}\| P_{\lambda_4}( u_1 \bar u_2 u_3)\|_{H^{s_1}}
\lesssim \| u_1 \bar u_2  u_3\|_{H^{s_0}},
\]
after which we can reuse the bounds in Proposition~\ref{p:ee-u}.

Finally, for the last term we separate variables,
and it remains to estimate
\[
\lambda_4^{s_1} t |a'(\lambda_1)-a'(\lambda_4)|
\| u_1 \bar u_2 u_3\|_{L^2}.
\]
We split the triple product with respect to angles,
\[
u_1 \bar u_2 u_3 = \chi_{\lambda_3}u_1 \bar u_2 u_3
+ (1-\chi_{\lambda_3})u_1 \bar u_2 u_3,
\]
and estimate the two terms separately.
For the first one we use \eqref{vf-ell-ll2} for $u_1$,
\[
\begin{aligned}
\lambda_4^{s_1}t |a'(\lambda_1)-a'(\lambda_4)|
\| \chi_{\lambda_3}  u_1 \bar u_2 u_3\|_{L^2}
\lesssim & \ \lambda_4^{s_1} t |a'(\lambda_1)-a'(\lambda_4)|
\| \chi_{\lambda_3}  u_1 \|_{L^2}\|u_2\|_{L^\infty} \|u_3\|_{L^\infty}
\\
\lesssim & \ \lambda_4^{s_1} t |a'(\lambda_1)-a'(\lambda_4)|
\frac{1}{\lambda_1^{s_1} t|a'(\lambda_1)-a'(\lambda_3)|}\| u_1 \|_{X}\|u_2\|_{L^\infty} \|u_3\|_{L^\infty}.
\end{aligned}
\]
If $\sigma \geq -1$ then the coefficient equals
\[
\frac{\lambda_4^{s_1} |a'(\lambda_1)-a'(\lambda_4)|}
{\lambda_1^{s_1} \lambda_3^{\sigma+1}} \lesssim 
\frac{\lambda_4^{s_1} (\lambda_1+\lambda_4)^{\sigma+1}}
{\lambda_1^{s_1}  \lambda_3^{\sigma+1}} \leq 1.
\]
Else, $\lambda_4 \lesssim \lambda_1$ therefore the coefficient equals
\[
\frac{\lambda_4^{s_1} \lambda_4^{\sigma+1}}
{\lambda_1^{s_1}  \lambda_1^{\sigma+1}}  \leq 1.
\]
For the second one we use \eqref{vf-ell-ll2} for $u_1$,
\[
\begin{aligned}
\lambda_4^{s_1}t |a'(\lambda_1)-a'(\lambda_4)|
\|(1- \chi_{\lambda_3})  u_1 \bar u_2 u_3\|_{L^2}
\lesssim & \ \lambda_4^{s_1} t |a'(\lambda_1)-a'(\lambda_4)|
\|u_1\|_{L^\infty} \|u_2\|_{L^\infty}\| (1-\chi_{\lambda_3})  u_3 \|_{L^2}
\\
\lesssim & \ \lambda_4^{s_1} t |a'(\lambda_1)-a'(\lambda_4)|
\frac{1}{\lambda_3^{s_1} t \lambda_3 a''(\lambda_3)|}\| u_1 \|_{X}\|u_2\|_{L^\infty} \|u_3\|_{L^\infty},
\end{aligned}
\]
and the coefficient is again easily verified to be $\leq 1$
by considering the same two cases as above.
\medskip

After this reduction, it remains to consider the case
when 
\begin{equation}
 \lambda_1 \ll \lambda_4 \lesssim \lambda_2 \approx \lambda_3, \quad \quad \sigma < 1.    
\end{equation} 
Here we separate the case $\sigma < -2$, where the threshold $\lambda_0$
plays a role. Precisely, if $\lambda_3 > \lambda_0$ then we can use \eqref{LP3} 
where $r$ has size $\lambda_4^{\sigma+1}$
in order to write schematically
\[
LP_{\lambda_4} (u_1 \bar u_2 u_3) = P_{\lambda_4} (u_1 \bar u_2 L u_3) + \lambda_4^{-1} P_{\lambda_4} (u_1 \bar u_2  u_3) + t\lambda_4^{\sigma+1} P_{\lambda_4}(u_1 \bar u_2 u_3).
\]
The first two terms are easy to estimate directly. So it remains to consider the third, where 
we estimate
\[
\begin{aligned}
t\lambda_4^{\sigma+1} \|P_{\lambda_4}(u_1 \bar u_2 u_3)\|_{H^{s_1}}
\lesssim & \ \lambda_4^{s_1+\sigma+1} \|u_1\|_{L^\infty}\|u_2\|_{L^\infty}
\|u_3\|_{L^2} 
\\
\lesssim & \ t \lambda_4^{s_1+\sigma+1} \lambda_3^{-s_0}\|u_1\|_{L^\infty}\|u_2\|_{L^\infty}
 \|u_3\|_{X}  
\\
\lesssim & \ t \lambda_3^{\sigma+2} \|u_1\|_{L^\infty}\|u_2\|_{L^\infty}
 \|u_3\|_{X}, 
\end{aligned}
\]
where $t \lambda_3^{\sigma+2} \leq t \lambda_0^{\sigma+2} = 1$.

From here on, we will assume that $\lambda_3 < \lambda_0$ in the case
$\sigma < -2$. Since $\lambda_1 \ll \lambda_4$, we can harmlessly move the $P_{\lambda_4}$
projection to the product $\bar u_2 u_3$, and work with 
\[
 v := u_1 P_{\lambda_4}(\bar u_2 u_3).
\]

To simplify matters, we note that within the 
$\lambda_4$ frequency region we have
\[
l(x,\xi) = x-x_4 + t O( \lambda_4^{\sigma+1}),
\]
and similarly at the operator level we get
\begin{equation}\label{L-last}
\| L v \|_{H^{s_1}}
\lesssim \lambda_4^{s_1}( \| (x-x_4) u_1 P_{\lambda_4} (\bar u_2 u_3) \|_{L^2}
+ t \lambda_4^{\sigma+1} \| u_1 P_{\lambda_4} (\bar u_2 u_3)\|_{L^2}). 
\end{equation}
We will rely on this bound 
for $\sigma \geq -1$. However, for 
$\sigma < -1$ we can process the first term further. We first move $x-x_4$ inside $P_{\lambda_4}$ at the expense of a mild commutator term, which is schematically 
written as 
\[
(x-x_4) u_1 P_{\lambda_4} (\bar u_2 u_3)
=  u_1 P_{\lambda_4} ((x-x_4)\bar u_2 u_3)
+ \lambda_4^{-1} u_1 P_{\lambda_4} (\bar u_2 u_3).
\]
The contribution of the $L^2$ norm of the commutator term can be harmlessly included into the second RHS term in \eqref{L-last},
using the upper bound $\lambda_4 \lesssim \lambda_0$ if $\sigma < -2$. On the other hand for the main term we can write
\[
(x-x_4) u_3 =  L u_3 + 
t R u_3, \qquad |r| \approx \lambda_4^{\sigma+1}.
\]
Since $\sigma < -1$, the contribution of the error term $R$ can also be included
into the second RHS term in \eqref{L-last}.
Finally, for the $Lu_3$ term we estimate 
directly 
\begin{equation}\label{L-help}
\lambda_4^{s_1} \| u_1 P_{\lambda_4} 
 u_2 L u_3\|_{L^2} \lesssim 
\left( \frac{\lambda_4}{\lambda_3}\right)^{s_1} \|u_1\|_{L^\infty} 
\| u_2\|_{L^\infty} \| u_3\|_{X},
\end{equation}
which is an acceptable contribution. 
We arrive at the following simplification of \eqref{L-last}, 
\begin{equation}\label{L-last+}
\| L v \|_{H^{s_1}}
\lesssim RHS\eqref{L-help} + t \lambda_4^{s_1+\sigma+1} \| u_1 P_{\lambda_4} (\bar u_2 u_3)\|_{L^2}, \qquad \sigma < -1.
\end{equation}
\medskip

At this point we would like to consider angular localizations for the triple 
product $ v = u_1 P_{\lambda_4} (\bar u_2 u_3)$, centered on the angle associated to $\lambda_3$. This angular region has size 
$t \lambda_3 a''(\lambda_3) = t \lambda_3^{\sigma+1}$, whereas $v$
has frequency $\lambda_4$. So, by the uncertainty principle, this localization 
is meaningful only if 
\begin{equation}\label{t-uncert}
  t \lambda_3^{\sigma+1} \lambda_4 \gtrsim  1.    
\end{equation}
This constraint is nontrivial only if 
$\sigma < -1$. We dispense with the 
complementary range by estimating directly the second RHS term in \eqref{L-last+}
as follows:
\begin{equation}\label{direct}
\begin{aligned}
t \lambda_4^{s_1+\sigma+1} \| u_1 P_{\lambda_4} (\bar u_2 u_3)\|_{L^2}
&\lesssim t \lambda_4^{s_1+\sigma+1} \lambda_3^{-s_0} \| u_1 \|_{L^\infty} \| u_2 \|_{X} \|u_3\|_{L^\infty} \\
&= (t \lambda_3^{\sigma+1} \lambda_4)
\lambda_4^{s_1+\sigma} \lambda_3^{-s_0-\sigma -1}  \| u_1 \|_{L^\infty} \| u_2 \|_{X} \|u_3\|_{L^\infty}  ,
\end{aligned}
\end{equation}
where all the factors on the right are 
$\lesssim 1$ given our choice of $s_0$ and $s_1$. We assume \eqref{t-uncert} from here 
on.
\medskip

We are now ready to localize $v$ using the angular cutoff $\chi_{\lambda_3}$ associated to frequency $\lambda_3$ waves.
It is easier to first consider the contribution of $(1-\chi_{\lambda_3})v$. 
One difficulty we encounter is that we need to commute this localization 
with $P_{\lambda_4}$,
\[
(1-\chi_{\lambda_3})P_{\lambda_4} = (1-\chi_{\lambda_3}) P_{\lambda_4}(1-\tchi_{\lambda_3})+ R,
\]
where the error $R$ has size 
\[
\| R\|_{L^2 \to L^2} \lesssim \left(\frac{1}{t \lambda_4 \lambda_3^{\sigma+1}}\right)^{N}.
\]
Here, if $\sigma \geq -1$ then we get $t^{-N}$ and the $R$ bound becomes straightforward. Otherwise \eqref{t-uncert}
holds so we can simply add the $R$ 
bound to the computation in \eqref{direct}.

Hence we are left with the bound for the contribution of the expression 
\[
u_1 P_{\lambda_4}(\bar u_2 (1-\chi_{\lambda_3})u_3).
\]
to either \eqref{L-last} (for $\sigma \geq -1$) or \eqref{L-last+} (for $\sigma < -1$). This is 
\[
I_1 = \lambda_4^{s_1}( \| (x-x_4) u_1 P_{\lambda_4} (\bar u_2 (1-\chi_{\lambda_3}) u_3) \|_{L^2}
+ t \lambda_4^{\sigma+1} \| u_1 P_{\lambda_4} (\bar u_2 (1-\chi_{\lambda_3}) u_3)\|_{L^2}).
\]
Here we harmlessly commute $x-x_4$ inside $P_{\lambda_4}$, modulo a mild error term which is controlled by the second term on the right. Then we use Proposition~\ref{p:vf-ell}
to estimate
\[
\| (1-\chi_{\lambda_3})(x-x_3) u_3 \|_{L^2} \lesssim
\lambda_3^{-s_1} \|u_3\|_{X}.
\]
 Bounding the other two factors 
in $L^\infty$, this yields
\[
I_1
\lesssim \lambda_4^{s_1} \lambda_3^{-s_1} \sup_{x\not \in A_3} \frac{|x-x_4| + t \lambda_4^{\sigma+1}}{|x-x_3|}\|u_3\|_{X} \|u_1\|_{L^\infty} \|u_2\|_{L^\infty}.
\]
The supremum is attained when $x$ is closest to $x_3$, i.e. when 
$|x-x_3| \approx t \lambda_3 a''(\lambda_3)$, in which case we get 
the coefficient
\[
\lambda_4^{s_1} \lambda_3^{-s_1} \frac{|a'(\lambda_4) - a'(\lambda_3)|}{\lambda_3 a''(\lambda_3)}. 
\]
If $\sigma \geq -1$ this gives 
\[
\lambda_4^{s_1} \lambda_3^{-s_1} \leq 1.
\]
If $\sigma < -1$ we get instead
\[
\lambda_4^{s_1+\sigma+1} \lambda_3^{-s_1-\sigma-1} \leq 1,
\]
both of which suffice.

\bigskip

Finally, we consider the most difficult case, where we 
estimate the contribution of  $\chi_{\lambda_3} v = \chi_{\lambda_3} u_1 P_{\lambda_4}(\bar u_2 u_3)$,
namely 
\[
\begin{aligned}
I_2 = & \  \lambda_4^{s_1}( \| (x-x_4) \chi_{\lambda_3} u_1 P_{\lambda_4} (\bar u_2 u_3) \|_{L^2}
+ t \lambda_4^{\sigma+1} \| \chi_{\lambda_3} u_1 P_{\lambda_4} (\bar u_2 u_3)v\|_{L^2}) \\ \approx & \  
t\lambda_4^{s_1} |a'(\lambda_4) - a'(\lambda_3)| \| \chi_{\lambda_3} u_1 P_{\lambda_4} (\bar u_2 u_3)\|_{L^2}.
\end{aligned}
\]
Here we can apply the bound \eqref{vf-ell-llinf} for 
$u_1$ to get
\begin{equation}\label{u1-low}
\| \chi_{\lambda_3} u_1\|_{L^\infty} \lesssim \frac{\lambda_1^{-s_1+\frac12}}{t |a'(\lambda_1) - a'(\lambda_3)|} \|u_1\|_{X}.
\end{equation}
On the other hand, for $\bar u_2 u_3$ we compute
\[
 t\partial_x (\bar u_2 u_3) = \overline{\tL u_2} u_3 + \bar u_2 \tL u_3,
\]
which allows us to estimate
\[
\| P_{\lambda_4} (\bar u_2 u_3) \|_{L^2} \lesssim \frac{\lambda_3^{-s_1}}{t\lambda_4 a''(\lambda_3)}(\| u_2\|_{X} \|u_3\|_{L^\infty}
+ \|u_2\|_{L^\infty} \|u_3\|_{X}).
\]
Combining the last two bounds, we arrive at
\[
I_2
\lesssim \lambda_4^{s_1-1}{|a'(\lambda_4)-a'(\lambda_3)|} \frac{\lambda_1^{-s_1+\frac12}}{|a'(\lambda_1) - a'(\lambda_3)|}\frac{\lambda_3^{-s_1}}{t a''(\lambda_3)}\|u_1\|_{X}
(\| u_2\|_{X} \|u_3\|_{L^\infty}
+ \|u_2\|_{L^\infty} \|u_3\|_{X}).
\]
Now we examine the coefficient in front. If $ \sigma \geq -1$ then we obtain
\[
\lambda_4^{s_1-1} \lambda_1^{-s_1+\frac12} t^{-1} \lambda_3^{-s_1-\sigma} \leq \lambda_4^{-\frac12} t^{-1},
\]
which is more than sufficient. 

However, if $\sigma < -1$ then we get instead
\[
I_2 \lesssim \lambda_4^{s_1+\sigma} \lambda_1^{-s_1-\sigma-\frac12} t^{-1} \lambda_3^{-s_1-\sigma}\|u_1\|_{X}
(\| u_2\|_{X} \|u_3\|_{L^\infty}
+ \|u_2\|_{L^\infty} \|u_3\|_{X}),
\]
which is unsatisfactory since the power of the high frequency $\lambda_3$ is positive. To rectify this, we use again \eqref{u1-low}
but estimate $u_3$ directly in $L^2$ to get
\[
\begin{aligned}
I_2 \lesssim & \ \lambda_4^{s_1}{|a'(\lambda_4)-a'(\lambda_3)|} \frac{\lambda_1^{-s_1+\frac12}}{|a'(\lambda_1) - a'(\lambda_3)|} \lambda_3^{-s_0} \| u_1\|_{X} \|u_2\|_{L^\infty} \|u_3 \|_{H^{s_0}}
\\
\lesssim & \ \lambda_4^{s_1+\sigma+1} \lambda_1^{-s_1-\sigma -\frac12}\lambda_3^{-s_0} \| u_1\|_{X} \|u_2\|_{L^\infty} \|u_3 \|_{H^{s_0}}.
\end{aligned}
\]
This has a negative power of $\lambda_3$
but insufficient time decay. Combining the two bounds 
we arrive at 
\[
I_2 \lesssim \left(\frac{\lambda_4}{\lambda_1}\right)^{s_1+\sigma+\frac12}
\min\left\{ t^{-1} \lambda_4^{\frac12} \lambda_3^{-s_0},
\lambda_4^{-\frac12} \lambda_3^{-s_1-\sigma}\right\} 
\|u_1\|_{X}
(\| u_2\|_{X} \|u_3\|_{L^\infty}
+ \|u_2\|_{L^\infty} \|u_3\|_{X}).
\]
Here the first exponent is negative $s_1+\sigma+\frac12 < 0$, 
and thus favourable. In the second factor, balancing exactly at the middle would yield the factor
\[
t^{-\frac12} \lambda_3^{\frac{-s_0-s_1-\sigma}2},
\]
with a favourable  negative power of $\lambda_3$ but a marginally 
insufficient power of $t$. But unbalancing this slightly suffices 
in order to improve the power of $t$ while maintaining a negative power for $\lambda_3$. 

This concludes the proof of Lemma~\ref{l:LQunbal}.

\end{proof}

The proof of Proposition~\ref{p:ee-Lu} is now also concluded.
\end{proof}

\subsection{Wave packets and the asymptotic profile}

For each admissible velocity $v \in J = a'(\R)$
we define the associated wave packet $\uu_v$
using the same formula \eqref{wp-def} as in the model case. Then the associated 
asymptotic profile $\gamma(t,v)$ can be defined exactly as before, following \eqref{def-gamma} but as a function 
\begin{equation}\label{domain-gamma}
\gamma: D = J \times \R^+ \to \mathbb C.
\end{equation}

If we consider velocities in the dyadic range $v \in J_\lambda$ then the spatial localization 
scale for the associated wave packet $\uu_v$ is
\[
\delta x \approx t^\frac12 (a''(\lambda))^\frac12.
\]
It is instructive to compare this scale with the size of the spatial region $\tilde J_\lambda$ associated to frequency $\lambda$, which is
\[
|\tilde J_\lambda| = t \lambda a''(\lambda).
\]
It is meaningful to define our asymptotic profile 
only if this dominates the wave packet scale,
\[
\delta x \lesssim   |\tilde J_\lambda|.
\]
This is equivalent to 
\[
t^\frac12 (a'')^\frac12 \lesssim t
\lambda a'' \quad  \Longleftrightarrow \quad t \gtrsim 
(\lambda^2 a''(\lambda))^{-1}.
\]
This is nontrivial only in the Klein-Gordon case $\sigma < -2$, where it can be rewritten in the 
form $\lambda \lesssim \lambda_0$,
which is the same threshold we have encountered before. Hence, from here on, in the case 
$\sigma < -2$ we will restrict $\gamma$ to a smaller set. Precisely, in this case we will redefine $D$
as
\begin{equation}\label{domain-gamma-low}
  D = \bigcup_\lambda \{ (v,t) \subset J \times \R^+; 
v \in J_\lambda, \ \  t   \gtrsim 
(\lambda^2 a''(\lambda))^{-1}\}.
\end{equation}

\bigskip

The first step in our study of the asymptotic 
profile $\gamma$ is to obtain bounds for it in terms of the $X$ norm of $u$.

\begin{lemma}\label{l:gamma}
Let $t \geq 1$, and $u \in X$ be a function at time 
$t$. Then within $D$ we have the bounds
\begin{equation}\label{gamma-L2}
\| \gamma \|_{L^2_v(J_\lambda)} \lesssim
(\lambda^{-s_0}+(t a''(\lambda) \lambda^2)^{-N})  \|u\|_{X},
\end{equation}
\begin{equation}\label{gamma-inf}
\| \gamma \|_{L^\infty(J_\lambda)}  \lesssim
(\lambda^{-\frac{\delta}{2}} +(t a''(\lambda) \lambda^2)^{-N})\|u\|_{X},
\end{equation}
\begin{equation}\label{gamma-dv}
\|\partial_v \gamma \|_{L^2_v(J_\lambda)} \lesssim
(\lambda^{-s_1-\sigma}+(t a''(\lambda) \lambda^2)^{-N}) \|u\|_{X}.
\end{equation}
\end{lemma}
We remark that the term $(t a''(\lambda) \lambda^2)^{-N})$ is only relevant in the case
$\sigma \leq -2$. Precisely, if $\sigma = -2$ then it gives
$t^{-N}$, and if $\sigma < -2$ then it gives
$(\lambda/\lambda_0)^N$.

\begin{proof}
Here we use the fact that, for $v \in J_\lambda$, our wave packet $\uu_v$ is essentially localized at frequency $\lambda$. Precisely, we can represent it 
as 
\[
\uu_v = (a''(\lambda))^{-\frac12} \chi(v,y)   e^{i \xi_v x}, \qquad y = \frac{x-vt}{\sqrt{t a''(\lambda)}},
\]
with $\chi$  Schwartz in $y$, uniformly in $v$.
This allows us to obtain favourable bounds for the portion of $\uu_v$ away from frequency $\lambda$,
\[
| \partial_x^{k} P_{\neq \lambda} \uu_v |
\lesssim_{k,N} (a''(\lambda))^{-\frac12} (1+|y|)^{-N} \lambda^k (t a''(\lambda) \lambda^2)^{-N},
\]
where one can distinguish three separate cases:

a) $\sigma > -2$, where we get an arbitrarily large gain,
\[
| \partial_x^{k} P_{\neq \lambda} \uu_v |
\lesssim_{k,N} \lambda^{-N} (1+|y|)^{-N} t^{-N}.
\]

b) $\sigma = -2$, where we only have the gain in time,
\[
| \partial_x^{k} P_{\neq \lambda} \uu_v |
\lesssim_{k,N} (a''(\lambda))^{-\frac12} (1+|y|)^{-N} t^{-N}.
\]

c) $\sigma < -2$, where the gain depends on the distance to $\lambda_0$,
\[
| \partial_x^{k} P_{\neq \lambda} \uu_v |
\lesssim_{k,N} (a''(\lambda))^{-\frac12} (1+|y|)^{-N} \lambda^k (\lambda_0/\lambda)^{-N}.
\]
We now use this in order to prove the three bounds in the Lemma.
\medskip

\emph{Proof of  \eqref{gamma-L2}:}
We separate frequencies $\lambda$ and frequencies 
away from $\lambda$, 
\[
\gamma(t,v) = \la u_\lambda, \uu_v \ra 
+ \la u, P_{\neq \lambda} \uu_v \ra.
\]
For the first inner product we use 
Young's inequality to get
\[
\| \la u_\lambda, \uu_v \ra \|_{L^2_v} \lesssim 
\|u_\lambda\|_{L^2_x},
\]
where we lose a $t^{\frac12}$ factor from the 
$L^1_x$ norm of $\uu_v$ but we regain it from the change of coordinates from $x$ to $x/t$.
For the second product, on the other hand, we take advantage of the rapid decay in the above bounds
for $P_{\neq \lambda} \uu_v$. In the nontrivial range 
$\sigma \leq -2$ we have $s_0 > 0$, so the worst contribution 
comes from frequencies $\lesssim 1$ in $u$.

\bigskip
\emph{Proof of  \eqref{gamma-inf}:}
Here we use instead the pointwise bound 
\eqref{vf-point-lambda} for $u$. The main contribution 
comes from $u_\lambda$ via Young's inequality,
while the other frequencies only contribute
a rapidly decaying tail, as above.

\bigskip
\emph{Proof of  \eqref{gamma-dv}:}
We have 
\[
\gamma_v(t,v) = \la u, \partial_v \uu_v\ra.
\]
For $\partial_v \uu_v$ we use the representation 
in Lemma~\ref{l:dv-uv}, to write the above expression as 
\[
\gamma_v(t,v) = \la Lu, \uu_v^2\ra +
\la u, \rr_v\ra .
\]
As above, the leading contribution comes from $u_\lambda$ where we can use directly Young's inequality. 
\end{proof}

Next we compare the asymptotic profile with the exact solution, working in the same region.

\begin{lemma}\label{l:u-gamma}
Suppose  $v \in J_\lambda$, with 
the additional restriction $\lambda < \lambda_0$
in the case when $\sigma < -2$. Then we have 
\begin{equation}
|\gamma(t,v) - t^{\frac12} u_\lambda(t,vt) e^{-it\phi(v)}| \lesssim
(t \lambda^2 a''(\lambda))^{-\frac14} 
\|u\|_{X}.
\end{equation}
\end{lemma}

\begin{proof}
Since $\lambda < \lambda_0$, the contribution 
of $u_{\neq \lambda}$ to $\gamma$ has size 
$(t \lambda^2 a''(\lambda))^{-N}$ and may be neglected.

Next we consider the contribution of $u_\lambda$ to $\gamma$, which generates the error
\[
r(t,v) = \la \uu_v, u_\lambda \ra - t^{\frac12} u_\lambda(t,vt) e^{-it\phi(v)}.
\]
Here the scales are fixed, 
so we can directly apply the argument in Section~\ref{s:bd-gamma}, Proposition~\ref{p:u-gamma} to get the error bound
\[
| r(t,v)| \lesssim  \lambda^{-s_1} t^{-\frac14} a''(\lambda)^{-\frac34} \|u\|_{X}
\lesssim  t^{-\frac14} \lambda^{-\frac12} (a'')^{-\frac14} \|u\|_{X},
\]
which is exactly as needed.
\end{proof}

Finally, we show that  $\gamma$ is a good approximate solution for the asymptotic equation,
\begin{equation}\label{gamma-asympt-re}
    \dot \gamma(t,v) = i q(\xi_v,\xi_v,\xi_v) t^{-1} \gamma(t,v) |\gamma(t,v)|^2 
+ f(t,v),
\end{equation}
where $f$ satisfies favourable bounds:

\begin{lemma}\label{l:asympt-err}
Suppose  $v \in J_\lambda$ and that, in addition, 
$\lambda < \lambda_0$ if $\sigma < -2$.
Then the error $f$ satisfies the 
uniform bound
\begin{equation}
|f(t,v)| \lesssim
(t^{-1} (t \lambda^2 a'')^{-\frac14} + 
t^{-1-\frac{\delta}4}\lambda^{-\frac{\delta}{4}} )
t^{C^2 \epsilon^2}, \qquad \sigma \leq -2, 
\quad t \lambda^2 a'' \geq 1,
\end{equation}
respectively 
\begin{equation}
|f(t,v)| \lesssim
t^{-1-\frac{\delta}4}\lambda^{-\frac{\delta}{4}} 
t^{C^2 \epsilon^2}, \qquad \sigma >  -2.
\end{equation}

\end{lemma}

\begin{proof}
We can write
\[
f(t,v) =  \langle u, (i \partial_t - A(D)) \uu_v \rangle + 
\langle  Q(u,\bar u,u), \uu_v \rangle : = f_1+f_2,
\]
and estimate each term separately.

For $f_1$ we use Lemma~\ref{l:Puu} to write
\[
(i \partial_t - A(D)) \uu_v = t^{-\frac32}( L \uu^1_v + \rr^1_v ),
\]
where
\[
\uu^1_v \approx (a'')^{-\frac12} \uu_v, \qquad \rr_v \approx \lambda^{-1} (a'')^{-\frac12} \uu_v.
\]
Hence we can use Holder's inequality to bound
\[
|f_1(t,v)| \lesssim t^{-\frac54} (a'')^{-\frac34}  \lambda^{-s_1} \| u\|_{X}
\lesssim t^{-1} (t a'' \lambda^2)^{-\frac14}\| u\|_{X}.
\]

Next we consider $f_2$, where we use the balanced/unbalanced decomposition of $Q$.
The contribution of the unbalanced part $Q^{unbal}$ is placed
in $f$ using the bound and \eqref{Q-L2-from-Lh}. 

It remains to consider the balanced component of $Q$. Furthermore, in view of the frequency localization of $\uu_v$ at frequency $\lambda$,  it suffices to consider 
the balanced component of $Q$ localized to frequency $\lambda$.

 Here we go through two stages, exactly as in the similar argument in the model case:

a) Replace $u_\lambda$ by $t^{-\frac12} \gamma(t,v) \chi_\lambda e^{it\phi} $, with errors controlled by Lemmas~\ref{l:u-gamma} and \ref{l:gamma}.

b) Replace $Q_\lambda(\chi_\lambda e^{it\phi},\chi_\lambda e^{-it\phi},\chi_\lambda e^{it\phi})$ by $t^{-\frac32} \chi_\lambda^3 e^{it\phi} q(\xi_v,\xi_v,\xi_v)$.

\end{proof}

\subsection{ Conclusion}
Our remaining objective is to recover our bootstrap assumption, and show that
we have the better bound
\begin{equation}
\| \la D\ra^{\frac{\delta}4} u \|_{L^\infty} \lesssim \epsilon    .
\end{equation}
We consider separately each dyadic component $u_\lambda$, for which we seek to show that 
\begin{equation}\label{want}
\| u_\lambda \|_{L^\infty}   \lesssim \epsilon t^{-\frac12} \lambda^{-\frac{\delta}3}.
\end{equation}
On the other hand, from the vector field bound \eqref{vf-point-lambda}
and the energy estimates \eqref{energy+}
we have
\begin{equation}\label{have}
\| u_\lambda \|_{L^\infty}   \lesssim \epsilon t^{-\frac12+C^2 \epsilon^2} \lambda^{-\frac{\delta}2}.
\end{equation}
Here $\epsilon$ is sufficiently small, so in particular we can assume
that $\epsilon \ll \delta$. Hence, the desired conclusion \eqref{have}
follows provided that 
$t \lesssim \lambda^{N}$,
where the large constant $N$ can be chosen arbitrarily.
It remains to consider the complementary region $t \gtrsim \lambda^N$.
We remark that in the case when $\sigma < -2$, this region lies entirely within $D$, so in particular it ensures that $\lambda < \lambda_0$.

We now divide and conquer depending on the spatial location:

a) Outside of the region $\tilde J_\lambda$, we use the elliptic bound \eqref{vf-ell-llinf}.
This yields 
\[
\|u_\lambda\|_{L^\infty(\tilde J_\lambda^C)} \lesssim 
\frac{\lambda^{-s_1+\frac12}}{t \lambda a''(\lambda)} t^{C^2 \epsilon^2}
\lesssim \frac{1}{(t \lambda^2 a''(\lambda))^\frac12} t^{C^2 \epsilon^2},
\]
which suffices if $t \gtrsim \lambda^N$.

b) It remains to bound $\chi_\lambda u_\lambda$. By Lemma~\ref{l:u-gamma}, this is equivalent to showing that our asymptotic function $\gamma$ satisfies a similar bound,
namely 
\begin{equation}
|\gamma(t,v)| \lesssim \epsilon \lambda^{-\frac{\delta}{2}}, \qquad v \in J_\lambda, \qquad  t \gtrsim \lambda^N  . 
\end{equation}
At this point it is natural to split into two cases:

A) $\sigma > -2$. Here we initialize $\gamma$ at $t = 1$, and use the asymptotic
equation \eqref{gamma-asympt-re} to bound $\gamma$ at later times.

B) $\sigma \leq - 2$. Here $\gamma$ is restricted to the set $D$, so for each velocity $v \in J_\lambda$ we initialize at times where $t \approx \lambda^N$,
using \eqref{have}, and propagate the bound using the asymptotic equation\eqref{gamma-asympt-re}.

\end{proof}

\bibliographystyle{plain}


\end{document}